\documentclass[12pt, reqno]{amsart}
\usepackage{mathrsfs}
\usepackage{amsfonts}
\usepackage{cite}
\usepackage{amsmath}
\usepackage{amsthm}
\usepackage{amssymb}
\usepackage{cases}
\usepackage{bm}
\numberwithin{equation}{section}
\newtheorem{thm}{Theorem}[section]

\newtheorem{lemma}{Lemma}[section]

\newtheorem{cor}{Corollary}[section]

\newtheorem{pro}{Proposition}[section]
\newtheorem{defi}{Definition}[section]
\newtheorem{theoremalph}{Theorem}

\usepackage[%
pdfstartview=FitH, %
colorlinks=true, %
linkcolor=blue, %
citecolor=red]{hyperref}

\begin{document}
\title[Dual curvature measures on convex functions] {Dual curvature measures on convex functions and  associated Minkowski problems}

\author{Niufa Fang}
\address{\parbox[l]{1\textwidth}{School of Mathematics, Hunan University, Changsha, Hunan 300071, China.}}
\email{fangniufa@hnu.edu.cn}

%
%
\author{Jiazu Zhou*}
\address{\parbox[l]{1\textwidth}{School of Mathematics and
Statistics, Southwest University, Chongqing 400715, China.\\
College of Science, Wuhan University of Science and Technology, Wuhan, Hubei 430081, China.}}
 \email{zhoujz@swu.edu.cn }

\subjclass[2010]{52A40, 26B25, 26D10} \keywords{Convex functions; functional dual curvature measures; dual quermassintegrals, functional dual Minkowski problem. }

\thanks{The first author is supported in part by NSFC (No.12001291).}
\thanks{*The corresponding author is supported in part by NSFC (No.12071318).}

\maketitle

\begin{abstract}
In this paper, the $q$-th dual curvature measure is extended  to  convex functions and the associated Minkowski problem is posed. A special case includes the $q$-th dual curvature measure of convex bodies  which defined by Huang, Lutwak, Yang and Zhang.   Existence for  the functional  dual   Minkowski problem is showed  when $q\leq0$ and the  uniqueness  part   is  obtained with some assumptions.
\end{abstract}

\vskip 0.5cm
\section{Introduction}
\vskip 0.3cm

The classical Brunn-Minkowski theory originated with the thesis of Brunn in 1887 and  is in its essential parts with the creation of Minkowski around the turn of the century.  Mixed volumes are important geometric concepts, which include quermassintegrals, surface area  measures, curvature measures and mixed area measures.  An interesting problem related to surface area  measures is the classical  Minkowski problem that  is crucial important  in the  Brunn-Minkowski theory:

\vskip 0.2cm
\emph{What are necessary and sufficient conditions
for a finite Borel measure $\mu$ on the unite sphere $S^{n-1}$  so that $\mu$ is the surface area measure of a convex body in $\mathbb{R}^n$?}

\vskip 0.2cm

The Minkowski problem was solved in multiple ways by Minkowski, Aleksandrov, and Fenchel and Jessen \cite{Minkowski,Schneider2014}. Landmark contributions to establishing regularity
for the Minkowski problem are due to  Lewy \cite{Lewy}, Nirenberg \cite{Nirenberg}, Pogorelov \cite{pogorelov1, pogorelov2},
Cheng and Yau \cite{ChengYau}, 
Caffarelli \cite{Caffarelli}, Guan and Ma \cite{GuanMa} and others.

\vskip 0.2cm

If the measure has a density $f$ with respect to the Lebesgue measure of the unit sphere $S^{n-1}$, the  Minkowski problem is equivalent to the study of solution to the following Monge-Amp$\grave{\text{e}}$re equation on the unit sphere
\begin{eqnarray*}
\det(h_{ij}+\delta_{ij}h)=f,
\end{eqnarray*}
where $h$ is the unknown function on $S^{n-1}$ to be found,   $h_{ij}$ is the covariant derivative of $h$ with  respect to an orthonormal frame on $S^{n-1}$ and $\delta_{ij}$ is the Kronecker delta.

\vskip 0.2cm

The $L_p$ Brunn-Minkowski theory, a generalization of the classical Brunn-Minkowski theory,   attracted  increasing interest in the recent years.  Lutwak \cite{lutwak1993} introduced the \emph{$L_p$ surface area measure},  a Borel measure  on convex body in $\mathbb{R}^n$  with  the origin in its interior. The $L_p$ surface area measure is an important concept in the $L_p$ Brunn-Minkowski theory, and it provides  an  integral representation  for  the $L_p$ mixed volume. By the $L_p$ cosine transform  of the $L_p$ surface area measure, Lutwak, Yang and Zhang \cite{LYZ2000} defined the $L_p$ Petty projection body and established  the well-known \emph{$L_p$ Petty projection inequality}.  Finding the necessary and sufficient conditions for a given measure to guarantee that it is the $L_p$ surface area measure is the existence problem called the \emph{$L_p$ Minkowski problem} posed in \cite{lutwak1993}. Solving the $L_p$ Minkowski problem requires solving a degenerate and singular Monge-Amp\`{e}re type equation on the unit sphere. The $L_p$ Minkowski problem has been solved for $p\geq1$ \cite{ChouWang,HugLYZ,lutwak1993,LYZ2004},  and  for the critical case $p<1$  \cite{blyz,stancu1,stancu2,Zhu1,Zhu2,Zhu3,Zhu5}. The solution of the $L_p$ Minkowski problem and the $L_p$ Petty projection inequality are the key tools used for establishing the \emph{$L_p$ affine Sobolev inequality} and analogues \cite{CianchiLYZ,HaberlSchuster,HaberlSchusterXiao,LYZ2002,Zhang1999}.

\vskip 0.2cm

A theory analogous to the Brunn-Minkowski theory was introduced in \cite{lutwak1975}. It demonstrates a remarkable duality in convex geometry, and thus is called the \emph{dual Brunn-Minkowski theory}. The important work of the duality between projection bodies and intersection bodies exhibited in \cite{lutwak1988}.  The dual Brunn-Minkowski theory has attracted  extensive attention since the intersection body helped achieving a major breakthrough in the solution of the celebrated Busemann-Petty problem. In the past decades, the dual Brunn-Minkowski theory has  a rapid growth \cite{Gardner1994,Gardner1994tams,Gardner1994ann,lutwak1990,Zhang1994,Zhang1999ann,ZhuZhouXu}.

\vskip  0.2cm

Very recently, Huang, Lutwak, Yang and Zhang \cite{HuangLYZ} introduced  \emph{dual curvature measures}. This new family of measures miraculously connects   the cone-volume measure and    Aleksandrov's integral curvature. The associated Minkowski problems of  dual curvature measures are called    \emph{dual Minkowski problems}.  Since the groundbreaking work of  Huang, Lutwak, Yang and Zhang,   the intrinsic partial differential equations that arise  within the dual Brunn-Minkowski theory which   wait a full 40 years after the birth of the dual theory to emerge.  
 Some  significant   breakthroughs of dual Minkowski problems have been achieved \cite{CHZ,BoroczkyHenkPollehn,blyzz,Gardner2018,HuangLYZ2,huangzhao,LiShengWang,LYZ2018,WangFangZhou,XY,Zhao1,Zhao2,ZhuYe2018}  but many cases   remain open. 

\vskip 0.2cm

Motivated by works of Huang, Lutwak, Yang and Zhang \cite{HuangLYZ}, we will introduce   the  \emph{functional  dual  curvature measures} on  the setting  of convex functions.  The new functional dual curvature measures include dual curvature measures introduced in \cite{HuangLYZ} as a special case.  All works in this paper are characterized as  the \emph{functional dual Brunn-Minkowski theory},  a new branch  in  convex geometric analysis.

\vskip 0.2cm

It is well known that the classical  isoperimetric inequality is equivalent to the Sobolev inequality. This equivalence   shows a close  connection between geometry and analysis. The geometric inequalities  and geometric problems of  functions received considerable attention. Actually an  analytic inequality  contains  more information than its  corresponding   geometric inequality.  During the past decades, many analytic  inequalities with geometric background  were found \cite{Alonso1,Alonso,A1,A2,C4,C1,C2,FXY,FXZZ,fangzhou,F1,F2,Lin,Milman2,Rotem2} basically from the viewpoint of integral and convex geometry.

\vskip 0.2cm

The algebraic structure (i.e., ``Minkowski sum" and ``scalar multiplication") on the set of convex bodies is the cornerstone in the classical Brunn-Minkowski theory. By using the \emph{infimal convolution} and  \emph{right scalar multiplication} of convex functions,  Colesanti and Fragal$\grave{\text{a}}$ \cite{Colesanti} introduced the  ``sum" and ``scalar multiplication"   of log-concave functions. Colesanti and Fragal$\grave{\text{a}}$'s  work belongs to    the \emph{functional Brunn-Minkowski theory}. We recall that the  infimal convolution and  the  right scalar multiplication. Let $\varphi,\psi :\mathbb{R}^n\to\mathbb{R}\cup\{+\infty\}$ be convex functions. The \emph{infimal convolution} of $\varphi$ and $\psi$   is defined by  \cite{Rockafellar},
\begin{eqnarray*}\label{}
\varphi\Box \psi(x)=\inf_{y\in\mathbb{R}^n}\{\varphi(x-y)+\psi(y)\} \quad \forall x\in\mathbb{R}^n,
\end{eqnarray*}
and the \emph{right scalar multiplication}, $\varphi t$,  is defined by,
\begin{eqnarray*}\label{}
(\varphi t)(x)=t\varphi\left(\frac{x}{t}\right), \quad\text{for}\quad t>0.
\end{eqnarray*}\vskip 0.2cm

In order  to study the  functional dual Minkowski problems, the crucial step is to properly define the functional  dual quermassintegrals. For a  non-negative convex function  $\varphi:\mathbb{R}^n\to\mathbb{R}\cup\{+\infty\}$  and $q\in \mathbb{R}$, the  $(n-q)$-th  dual quermassintegral $\widetilde{W}_{n-q}(\varphi)$ is defined by 
  \begin{eqnarray}
\widetilde{W}_{n-q}(\varphi)=\int_{\mathbb{R}^n}\varphi(x)^{-q}d\gamma_n(x),
\end{eqnarray}
 where $d\gamma_n(x)=(2\pi)^{-\frac{n}{2}}e^{-\frac{|x|^2}{2}}dx$ denotes the Gaussian probability measure.  We will prove  that the $(n-q)$-th dual quermassintegral, $\widetilde{W}_{n-q}(\varphi)$,  of the convex function $\varphi$ includes the $(n-q)$-th dual quermassintegral of a convex body.

For $q\in \mathbb{R}\setminus\{0\}$,   the $(n-q)$-th \emph{dual mixed quermassintegral},  $\widetilde{W}_{n-q}(\varphi,\psi)$,  of non-negative convex functions $\varphi$ and $\psi$  is defined by 
\begin{eqnarray}
\widetilde{W}_{n-q}(\varphi,\psi)=\frac{1}{q}\lim_{t\rightarrow0^{+}}\frac{\widetilde{W}_{n-q}(\varphi\square(\psi t))-\widetilde{W}_{n-q}(\varphi)}{t}.
\end{eqnarray}

The  Minkowski type inequality for the functional dual mixed quermassintegrals can be stated as follows:
\vskip 0.2cm
\textbf{Minkowski type inequality.} \emph{If $q\leq0$ and $\varphi,\psi$ are non-negative convex functions in $\mathbb{R}^n$, then \begin{eqnarray}\label{Minkowski inequality0}
\widetilde{W}_{n+1-q}(\varphi,\psi)\geq\frac{\widetilde{W}_{n+1-q}(\psi)^{\frac{1}{1-q}}
 -\widetilde{W}_{n+1-q}(\varphi)^{\frac{1}{1-q}}}{\widetilde{W}_{n+1-q}(\varphi)^{\frac{q}{1-q}}}
 +\widetilde{W}_{n+1-q}(\varphi,\varphi),
\end{eqnarray}
with equality  if and only if $\varphi(x)=\psi(x-x_0)$ for some $x_0\in \mathbb{R}^n$ when $q=0$, and $\varphi(x)=\psi(x)$ when $q<0$. }

It  should be pointed  out  that $\widetilde{W}_{n+1-q}(\varphi,\varphi)$ does not agree with $\widetilde{W}_{n+1-q}(\varphi)$ but  it  satisfies
\begin{eqnarray*}
q\widetilde{W}_{n-q}(\varphi,\varphi)
=(n-q)\widetilde{W}_{n-q}(\varphi)-\int_{\mathbb{R}^n}|x|^2
\varphi(x)^{-q}d\gamma_n(x),
\end{eqnarray*}
for $q\in \mathbb{R}\setminus\{0\}$.

If $q\leq 0$  and $\varphi,\psi$ satisfy some additional conditions (see details  in Theorem \ref{integral formula}),  we prove that the functional dual mixed  quermassintegral  has  the following integral representation: 
\begin{eqnarray}\label{another definition of dual curvature measure}
\frac{d}{dt}\widetilde{W}_{n+1-q}(\varphi\Box (\psi t))\Big|_{t=0^+}=\int_{\mathbb{R}^n}\psi^*(x)d\widetilde{C}_q(\varphi,x),
\end{eqnarray}
where $\psi^*$ is  the {\it Legendre transform} of $\psi$ (defined in Section \ref{Background}) and the Borel measure $\widetilde{C}_q(\varphi,\cdot)$ on $\mathbb{R}^n$ is defined by 
$$
\widetilde{C}_q(\varphi,\cdot)=(\nabla \varphi)_{\sharp}(\varphi^{-q}d\gamma_n).
$$
When limited to a subclass of convex functions, $\widetilde{C}_q(\varphi,\cdot)$ reduces to  the $q$-th dual curvature measure of a convex body. The $q$-th dual curvature measure of the convex function $\varphi$ is also called the \emph{functional $q$-th  dual curvature measure}. It  leads naturally to the following Minkowski problem involving the functional $q$-th  dual curvature measure.

\vskip 0.2cm
\textbf{The  functional dual Minkowski problem.}  \emph{Let $\mu$ be a  finite Borel measure on $\mathbb{R}^n$ and $q\in \mathbb{R}$. Find the necessary and sufficient conditions on $\mu$ so that it is the $q$-th dual curvature measure $\tau\widetilde{C}_q(\varphi,\cdot)$ of a convex function $\varphi$ in $\mathbb{R}^n$, where the constant $\tau\in \mathbb{R}$.}
\vskip 0.2cm

We will show that  when the measure $\mu$ has a density function $g:\mathbb{R}^n\rightarrow \mathbb{R}$, the partial differential equation for  the functional dual Minkowski problem is a  Monge-Amp\`{e}re type equation: 
\begin{eqnarray}
(2\pi)^{-\frac{n}{2}}f(\nabla f^*
)^{-q}e^{-\frac{|\nabla f^*|^2}{2}}\det(\nabla^2f^*)=g \quad \text{on}\quad \Omega,
\end{eqnarray}
where $f$ is the unknown convex function in $\mathbb{R}^n$ to be found, $\nabla f$ and $\nabla^2 f$ denote  respectively the gradient vector and the Hessian matrix of $f$ with respect to an orthonormal frame in $\mathbb{R}^n$, and $f^*$ denotes the Legendre transform of $f$ (see the detailed  definition in Section \ref{Background}), $\Omega=\{x\in \mathbb{R}^n: f(x)<+\infty\}\cap \{x\in \mathbb{R}^n:f^*(x)<+\infty\}$.

\vskip 0.2cm

We solve  the functional dual Minkowski problem when  $q\leq0$.

\vskip 0.3cm 
\textbf{Existence of solution to the functional dual Minkowski problem.}  \emph{Let $\mu$ be a non-zero finite Borel measure with $\int_{\mathbb{R}^n}|x|d\mu(x)<\infty$ on  $\mathbb{R}^n$  and $q\leq0$.   If $\mu$ is not supported in a lower-dimensional subspace, then there exists a constant $\tau\in\mathbb{R}$  and   a non-negative,  convex function $\varphi$ with finite  $\int_{\mathbb{R}^n}\varphi(x)^{1-q}d\gamma_n(x)$ such that $\mu(\cdot)=\tau\widetilde{C}_q(\varphi,\cdot)$.  Moreover, if $\varphi$ is essentially-continuous  then  $\widetilde{C}_q(\varphi,\cdot)$ is not supported in a lower-dimensional subspace}.
\vskip 0.3cm 
We remark that a convex function $\varphi : \mathbb{R}^n\rightarrow\mathbb{R}\cup\{+\infty\}$ is  \emph{essentially-continuous} if $\varphi$ is lower semi-continuous and if the set of points where $\varphi$ is discontinuous has zero $\mathcal{H}^{n-1}$-measure,  where $\mathcal{H}^{n-1}$ is the $(n-1)$-dimensional  Hasusdorff measure.

\vskip 0.2cm

By the  Minkowski type inequality (\ref{Minkowski inequality0}), we establish  the uniqueness of the solution to the   functional dual Minkowski problem.

\vskip 0.2cm

\textbf{Uniqueness of the solution to the functional dual Minkowski problem.} \emph{Let $q\leq0$ and let $\varphi_1,\varphi_2$ be non-negative convex functions in $\mathbb{R}^n$ with $\varphi_1(0)=\varphi_2(0)=0$. If there exist $c_1,c_2>0$ such that $\varphi_1^*-c_2\varphi_2^*$ and $\varphi_2^*-c_1\varphi_1^*$ are convex,   then $\widetilde{C}_{q-1}(\varphi_1,\cdot)=\widetilde{C}_{q-1}(\varphi_2,\cdot)$ implies that  $\varphi_1$ and $\varphi_2$ agree up to a translation when $q=0$, and $\varphi_1=\varphi_2$ almost everywhere when $q<0$.}

\vskip 0.2cm 
In \cite{HuangLYZ}, authors proved that the $0$-th dual curvature measure  is  the Aleksandrov's integral curvature, and the related Minkowski problem is called  Aleksandrov problem. Therefore, the existential part  can be viewed   as the solution to the \emph{functional Aleksandrov problem} when $q=0$. Similarly, the functional $n$-th  dual curvature measure, $\widetilde{C}_n(\varphi,\cdot)$, is the functional cone-volume measure,  and the corresponding Minkowski problem is the \emph{functional log-Minkowski problem}.
Unfortunately, the uniqueness part  does not imply  these two special cases.
In this paper, we only consider  the functional dual Minkowski problem for the case $q\leq0$.

%
%
%
%

~~\vskip 0.3cm
\section{Preliminaries}\label{Background}

\subsection{The setting of convex functions}
~~\vskip 0.3cm
A function  $\varphi: \mathbb{R}^n\rightarrow \mathbb{R}\cup \{+\infty\}$ is convex  if for every $x,y\in\mathbb{R}^n$ and $\lambda\in[0,1]$
$$
\varphi((1-\lambda)x+\lambda y)\leq (1-\lambda)\varphi(x)+\lambda \varphi(y).
$$
Let
$$
\text{dom}(\varphi)=\{x\in \mathbb{R}^n: \varphi(x)\in \mathbb{R}\}.
$$
We say that $\varphi$ is \emph{proper} if $\text{dom}(\varphi)\neq\emptyset$.
   The \emph{Legendre transform} of $\varphi$ is the convex function defined by
\begin{eqnarray}\label{Fenchel conjugate}
\varphi^*(y)=\sup_{x\in\mathbb{R}^n}\left\{\langle x,y\rangle-\varphi(x)\right\}\quad\quad\forall y\in\mathbb{R}^n.
\end{eqnarray}
Clearly, $\varphi(x)+ \varphi^*(y)\geq \langle x,y\rangle$ for all $x, y\in\mathbb{R}^n $. There is an equality if and only if  $\varphi(x)<+\infty$ and $y\in \partial\varphi(x)$,  the subdifferential  of $\varphi$ at $x$  define by
\begin{eqnarray}\label{subdifferential}
\partial\varphi(x)=\{z\in \mathbb{R}^n: \varphi(y)\geq\varphi(x)+\langle z,y-x\rangle \ \text{for all} \ y\in \mathbb{R}^n\}.
\end{eqnarray}
 Hence,
$$\varphi^*(\nabla \varphi(x))+\varphi(x)= \langle x,\nabla \varphi(x)\rangle.$$
The following elementary property was proved in \cite{Rockafellar},
\begin{eqnarray}\label{elementary property of Fenchel conjugate}
\varphi^*(0)=-\inf\varphi.
\end{eqnarray}

For convenient, we set
\begin{eqnarray*}\label{}
\mathcal{L}&=&\left\{\varphi:\mathbb{R}^n\rightarrow \mathbb{R}\cup \{+\infty\} \Big| \quad  \varphi \text{ proper, convex, }\lim_{|x|\rightarrow+\infty}\varphi(x)=+\infty\right\},
\end{eqnarray*}
and   $\mathcal{L}_o$,  all non-negative convex functions of the subset of $\mathcal{L}$.

 A  function $\varphi \in\mathcal{L}$ is  \emph{lower semi-continuous}, if the subset $\{x\in \mathbb{R}^n: \varphi(x)>t\}$ is an open set for any $t\in(-\infty,+\infty]$. The function $\varphi^*$ is always convex and lower semi-continuous.  If $\varphi$ is a lower semi-continuous  convex function, then  $\varphi^*$ is also a lower semi-continuous  convex function, and $\varphi^{**}=\varphi$.  It should be noted that  for a general  function $\varphi$ one always  has $\varphi^{**}\leq\varphi$.

For $\varphi, \psi\in\mathcal{L}$, the \emph{infimal convolution} is defined by
\begin{eqnarray}\label{infimal convolution}
\varphi\Box \psi(x)=\inf_{y\in\mathbb{R}^n}\{\varphi(x-y)+\psi(y)\} \quad \forall x\in\mathbb{R}^n,
\end{eqnarray}
and the \emph{right scalar multiplication}  is defined by,
\begin{eqnarray}\label{right scalar multiplication}
(\varphi t)(x)=t \varphi\left(\frac{x}{t}\right), \quad\text{for}\quad t>0.
\end{eqnarray}

A function $f:\mathbb{R}^n\rightarrow \mathbb{R}$ is \emph{log-concave} if $
f=e^{-\varphi}$ with  $\varphi\in\mathcal{L}$. Then we have  the following   Pr\'ekopa-Leindler inequality  \cite{Schneider2014}:

Let $0<t<1$ and  $f, g, h$ be non-negative integrable functions
in $\mathbb{R}^n$ satisfying

$$
h((1-t)x+ty)\geq f(x)^{1-t}g(y)^{t},
$$
for all $x,y\in\mathbb{R}^n$. Then
\begin{eqnarray}\label{Prekopa-Leindler inequality}
\int_{\mathbb{R}^n}h(x)dx \geq \left(\int_{\mathbb{R}^n}f(x)dx\right)^{1-t}\left(\int_{\mathbb{R}^n}g(x)dx\right)^t.
\end{eqnarray}
Equality holds in (\ref{Prekopa-Leindler inequality})
if and only if $f$ and $g$ are log-concave functions, and $f(x)=g(x-x_0)$ for some $x_0\in\mathbb{R}^n$.

\vskip 0.2cm

The following properties will be useful late \cite{Colesanti,Rockafellar}.

\begin{pro}\label{Colesanti's proposition 2.1}
Let $\varphi, \psi\in\mathcal{L}$. Then
\begin{enumerate}

\item  $(\varphi\square \psi)^*=\varphi^*+\psi^*$;

\item $(\varphi t)^*=t \varphi^*$, $t>0$.
\end{enumerate}
\end{pro}

\begin{theoremalph}\label{Rockafellar's result} \emph{\cite[Theorem 10.9]{Rockafellar}}
Let $C$ be a relatively open convex set, and let $f_1,f_2,\cdots,$ be a sequence of finite convex functions on $C$. Suppose that the real number  sequence $f_1(x),f_2(x),\cdots,$
is bounded for each $x\in C$. It is then possible to select a subsequence of  $f_1,f_2,\cdots,$ which converges uniformly  on  closed  bounded subsets of $C$ to some finite convex function $f$.
\end{theoremalph}

\vskip 0.3cm
~
\subsection{Dual curvature measures}
~~\vskip 0.3cm

Let $\mathbb{R}^n$ be  equipped with the usual Euclidean norm $|\cdot|$.  A subset $K$ of $\mathbb{R}^n$ is called a convex body if it is a compact convex set with non-empty interior.  The set of all convex bodies that contain the origin in their  interiors is denoted by $\mathcal{K}_o^n$.  For $x,y\in\mathbb{R}^n$, let $\langle x, y\rangle$ denote  the standard inner product of $x$ and $y$.
For a convex body $K$ in $\mathbb{R}^n$, we write $\partial K$ for the  boundary of $K$.

For $K, L\in\mathcal{K}_o^n$ and real  $t,s>0$, the Minkowski combination, $t K+s L$, is defined by
$$
t K+s L=\{t x+s y:x\in K,y\in L\}.
$$
The support function $h_K (\cdot)=
h(K, \cdot): \mathbb{R}^n \rightarrow \mathbb{R}$ of a convex body $K$ is defined by
$$
h(K,x)=\max\{\langle x,y\rangle: y\in K\},~~~~~~x\in \mathbb{R}^n.
$$
The convex body $K$ is uniquely determined by its support function $
h(K, \cdot)$. It is obvious that for $t
>0$, the support function of the convex body $t K =\{t x: x\in
K\}$ satisfies
\begin{eqnarray*}
h(t K,\cdot)=t h(K,\cdot).
\end{eqnarray*}
For  $K\in \mathcal{K}_o^n$,   its  \emph{radial function} $\rho(K,\cdot):\mathbb{R}^n\setminus\{o\}\rightarrow [0,\infty)$  is defined  by
\begin{eqnarray*}
\rho(K,x)=\max\{t\geq 0:t x\in K\}.
\end{eqnarray*}
The radial function is denoted also by $\rho_K(x)$.


For a convex body $K$, its \emph{gauge function} $\|\cdot\|_K$ is defined by
\begin{eqnarray}\label{}
\|x\|_K=\min\{t \geq 0:x\in t K\}.
\end{eqnarray}
The \emph{polar body}  $K^{\circ}$  of $K$ is defined  by
$$
K^{\circ}=\left\{x\in\mathbb{R}^n: \langle x, y\rangle\leq 1\quad\text{for all}\quad y\in K  \right\}.
$$
If $K\in \mathcal{K}_o^n$, then
\begin{eqnarray}\label{gauge function}
\|x\|_K=h_{K^{\circ}}(x)=\frac{1}{\rho_K(x)}.
\end{eqnarray}
It is obvious that
\begin{eqnarray}\label{the value of gauge function at the bounary point}
\|x\|_K=1\quad \text{whenever}\quad x\in \partial K.
\end{eqnarray}

Let $K\in\mathcal{K}_o^n$   and $q\in\mathbb{R}$.
The $q$-th \emph{dual curvature measure} $\widetilde{C}_q(K,\cdot)$ is defined by (see, e.g.,  \cite{HuangLYZ})
\begin{eqnarray}\label{q-dual curvature measure}
\widetilde{C}_q(K,\omega)=\frac{1}{n}\int_{\bm{\alpha}_K^*(\omega)}\rho_K^q(u)du,
\end{eqnarray}
for each Borel set $\omega\subset S^{n-1}$. Here   $\bm{\alpha}_K^*(\omega)$ denotes the set of directions $u\in S^{n-1}$ such that the boundary point $\rho_K(u)u$ belongs to $\nu_K^{-1}(\omega)$, where $\nu_K$ is the Gaussian map which defined on the boundary $\partial K$ of $K$, and $\nu_K^{-1}$ is the inverse Gaussian map. Equivalently,  for a bounded Borel  function $g:S^{n-1}\rightarrow \mathbb{R}$,
\begin{eqnarray}\label{the integral formal of q-dual curvature measure}
\int_{S^{n-1}}g(v)d\widetilde{C}_q(K,v)=\frac{1}{n}\int_{S^{n-1}}g(\alpha_K(u))\rho_K^q(u)du,
\end{eqnarray}
where $\alpha_K(u)=\nu_K(\rho_K(u)u)$. Moreover, if $K$ is strictly convex, then
\begin{eqnarray}\label{the integral formal of q-dual curvature measure2}
\int_{S^{n-1}}g(v)d\widetilde{C}_q(K,v)=\frac{1}{n}\int_{\partial K}g(\nu_K(x))h_K(\nu_K(x))|x|^{q-n}d\mathcal{H}^{n-1}(x),
\end{eqnarray}
where $\mathcal{H}^{n-1}$ is the $(n-1)$-dimensional Hausdorff measure. 
We note that the total measure of the $q$-th dual curvature measure  is the $(n-q)$-th dual quermassintegral $\widetilde{W}_{n-q}(K)$, i.e.,
\begin{eqnarray}\label{dual quermassintegral}
\widetilde{W}_{n-q}(K)=\frac{1}{n}\int_{S^{n-1}}\rho_K(u)^qdu.
\end{eqnarray}

It was proved in \cite{HuangLYZ} that the $n$-th dual curvature measure of a convex body is the cone volume measure of this convex body, while the zeroth dual curvature measure of a convex body is   Aleksandrov's integral curvature of the polar body of this body, i.e.,
\begin{eqnarray}\label{n-dual curvature measure}
\widetilde{C}_n(K,\cdot)&=&\frac{1}{n}h_K(\cdot)dS_K(\cdot),\\ \label{0-dual curvature measure}
\widetilde{C}_0(K,\cdot)&=&\mathcal{H}^{n-1}(\bm{\alpha}_{K^{\circ}}(\cdot)).
\end{eqnarray}
The associated Minkowski problems of the cone volume measure and the Aleksandrov's integral curvature measure   are called  \emph{log-Minkowski problem} (see, e.g.,  \cite{blyz,stancu1,stancu2,Zhu1}) and 
\emph{Aleksandrov problem} (see, e.g.,  \cite{Schneider2014}), respectively.

\section{Mixed dual quermassintegrals  of convex functions}\label{Dual quermassintegrals  for convex functions}

In this section, we will define the  dual quermassintegrals  for convex functions that  include the   dual quermassintegrals  for convex bodies. Inspired by works of \cite{HuangLYZ}, we  investigate the  functional  dual  mixed quermassintegrals.

~
\vskip 0.3cm
\subsection{Integral  formula of the functional dual mixed  quermassintegrals}
~
\vskip 0.3cm

Let $q\in\mathbb{R}\setminus\{0\}$, and $\varphi \in \mathcal{L}_o$. The $(n-q)$-th dual quermassintegral is defined as
\begin{eqnarray}\label{functional dual quermassintegrals}
\widetilde{W}_{n-q}(\varphi)=\int_{\mathbb{R}^n}\varphi(x)^{-q}d\gamma_n(x).
\end{eqnarray}
Let $q<n$ and $K\in \mathcal{K}_o^n$. By the polar coordinates of $x$, (\ref{dual quermassintegral}) and (\ref{gauge function}), we have
\begin{eqnarray*}
\widetilde{W}_{n-q}(\|x\|_{K^{\circ}})&=&\int_{\mathbb{R}^n}\|x\|_{K^{\circ}}^{-q}d\gamma_n(x)\\
&=&(2\pi)^{-\frac{n}{2}}\int_{0}^{+\infty}e^{-\frac{r^2}{2}}r^{n-q-1}dr\int_{S^{n-1}}\|u\|_{K^{\circ}}^{-q}du\\
&=&n(2\pi)^{-\frac{n}{2}}2^{\frac{n-q}{2}-1}\Gamma\left(\frac{n-q}{2}\right)\widetilde{W}_{n-q}(K).
\end{eqnarray*}
In this sense, our definition (\ref{functional dual quermassintegrals})  extended  dual quermassintegrals of convex bodies to the functional version.

For real $q\neq0$,  the normalized dual  quermassintegral $\overline{W}_{n-q}(\varphi)$  is defined as
\begin{eqnarray}
\overline{W}_{n-q}(\varphi)=
\left(\int_{\mathbb{R}^n}\varphi(x)^{-q}d\gamma_n(x)\right)^{\frac{1}{q}},
\end{eqnarray}
and,  for $q=0$, by
\begin{eqnarray}
\overline{W}_{n}(\varphi)=
\exp\left(-\int_{\mathbb{R}^n}\log\varphi(x)d\gamma_n(x)\right).
\end{eqnarray}

The  functional   dual  mixed quermassintegrals are  also introduced.
\begin{defi}\label{definion of mixed functional dual  quermassintegrals}
Let $q\in \mathbb{R}\setminus\{0\}$, and let $\varphi,\psi \in \mathcal{L}_o$. The $(n-q)$-th   dual mixed quermassintegral of $\varphi$ and $\psi$ is defined as
$$
\widetilde{W}_{n-q}(\varphi,\psi)=\frac{1}{q}\lim_{t\rightarrow0^{+}}\frac{\widetilde{W}_{n-q}(\varphi\square(\psi t))-\widetilde{W}_{n-q}(\varphi)}{t}.
$$
\end{defi}


In particular, when $\varphi=\psi$, we have

\begin{lemma}
Let $q\in \mathbb{R}\setminus\{0\}$, and let $\varphi \in \mathcal{L}_o$. If  $\widetilde{W}_{n-q}(\varphi)$ is finite, then
\begin{eqnarray*}
q\widetilde{W}_{n-q}(\varphi,\varphi)
=(n-q)\widetilde{W}_{n-q}(\varphi)-\int_{\mathbb{R}^n}|x|^2
\varphi(x)^{-q}d\gamma_n(x).
\end{eqnarray*}

\end{lemma}
\begin{proof}
From the definitions of infimal convolution (\ref{infimal convolution}) and  right scalar multiplication (\ref{right scalar multiplication}), we have
\begin{eqnarray*}
\varphi\square[\varphi t](x)&=&\inf_{y\in \mathbb{R}^n}\left\{\varphi(x-y)+t\varphi\left(\frac{y}{t}\right)\right\}\\
&\leq& (1+t)\varphi\left(\frac{x}{1+t}\right),
\end{eqnarray*}
and the convexity of $\varphi$ deduces
\begin{eqnarray*}
\frac{1}{1+t}\varphi(x-y)+\frac{t}{1+t}\varphi\left(\frac{y}{t}\right)\geq\varphi\left(\frac{x}{1+t}\right),
\end{eqnarray*}
for any $x,y\in\mathbb{R}^n$ and $t>0$. This means  that  $\varphi\square(\varphi t)=\varphi(1+t)$.  A direct   calculation yields
\begin{eqnarray*}
&&\frac{\widetilde{W}_{n-q}(\varphi\square(\varphi t))-\widetilde{W}_{n-q}(\varphi)}{t}\\
&&\quad\quad=\frac{(2\pi)^{-\frac{n}{2}}}{t}\left[\int_{\mathbb{R}^n}(1+t)^{-q}\varphi \left(\frac{x}{1+t}\right)^{-q}e^{-\frac{|x|^2}{2}}dx-
\int_{\mathbb{R}^n}\varphi(x)^{-q}e^{-\frac{|x|^2}{2}}dx\right]\\
&&\quad\quad=\frac{(2\pi)^{-\frac{n}{2}}}{t}\left[\int_{\mathbb{R}^n}(1+t)^{n-q}\varphi \left(x\right)^{-q}e^{-\frac{(1+t)^2|x|^2}{2}}dx-
\int_{\mathbb{R}^n}\varphi(x)^{-q}e^{-\frac{|x|^2}{2}}dx\right]\\
&&\quad\quad=(2\pi)^{-\frac{n}{2}}\left[\frac{(1+t)^{n-q}-1}{t}\right]\int_{\mathbb{R}^n}\varphi \left(x\right)^{-q}e^{-\frac{(1+t)^2|x|^2}{2}}dx\\
&&\quad\quad\quad+\ (2\pi)^{-\frac{n}{2}}\int_{\mathbb{R}^n}\varphi(x)^{-q}e^{-\frac{|x|^2}{2}}\left[\frac{e^{-\frac{(t^2+2t)|x|^2}{2}}-1}{t}\right]dx.
\end{eqnarray*}
By the monotone convergence theorem, let $t\rightarrow 0^{+}$, we obtain
\begin{eqnarray*}
q\widetilde{W}_{n-q}(\varphi,\varphi)
&=&(n-q)\widetilde{W}_{n-q}(\varphi)-\int_{\mathbb{R}^n}|x|^2
\varphi(x)^{-q}d\gamma_n(x).
\end{eqnarray*}
\end{proof}

Let
\begin{eqnarray*}\label{}
\mathcal{L'}&=&\left\{\varphi\in \mathcal{L} \Big| \quad  \varphi \in \mathcal{C}_{+}^2, \ \lim_{|x|\rightarrow+\infty}\frac{\varphi(x)}{|x|}=+\infty\right\}.
\end{eqnarray*}
Here $\mathcal{C}_{+}^2$ used for the set of functions whose Hessian matrix are positive definite at each point.   Let $\mathcal{L}_o'$ denote the set of the non-negative  convex function $\varphi\in \mathcal{L'}$ and  there exists  constants $a,b>0$ such that $\varphi(x)\leq b|x|^{1+a}$ when $|x|$ is big enough for $x\in\mathbb{R}^n$.

In order to give the integral representation  of the functional dual mixed   quermassintegral  we  need the following  three  lemmas.

\begin{lemma}\label{Colesantis lemma2.5}
\emph{\cite[Lemma 2.5]{Colesanti}} Let $\varphi\in \mathcal{L}$. There exist constants
$a, b \in \mathbb{R}$ with $a > 0$ such that
$$
\varphi(x)>a |x|+b,
$$
for $x\in \mathbb{R}^n$.
\end{lemma}

\begin{lemma}\label{Colesanti's lemma 4.10} \emph{\cite[Lemma 4.10]{Colesanti}}
Let $\varphi,\psi \in \mathcal{L'}$ and, for any $t>0$, set $\varphi_t=\varphi\Box(\psi t)$. Assume that $\psi(0)=0$, then

\begin{enumerate}
\item  $\varphi_{1}(x)\leq\varphi_{t}(x)\leq\varphi_{}(x),\quad \forall x\in \mathbb{R}^n,\ \forall t\in[0,1];$

\item  $\forall x\in\{x\in \mathbb{R}^n:\varphi(x)<+\infty\}$,\quad $\lim_{t\rightarrow 0^{+}}\varphi_t(x)=\varphi(x);$

\item  $\forall E\subset\{x\in \mathbb{R}^n:\varphi(x)<+\infty\}$, \quad $\lim_{t\rightarrow 0^{+}}\nabla\varphi_t(x)=\nabla\varphi(x)$ uniformly on $E$.
\end{enumerate}
\end{lemma}

\begin{lemma}\label{Colesanti's lemma 4.11}\emph{\cite[Lemma 4.11]{Colesanti}} Let $\varphi,\psi \in \mathcal{L'}$ and, for any $t>0$, set $\varphi_t=\varphi\Box(\psi t)$. Assume that $\psi(0)=0$, then for $x\in {\rm int}(\{x\in \mathbb{R}^n:\varphi(x)<+\infty\})$
$$
\frac{d}{dt}\varphi_t(x)=-\psi^*(\nabla\varphi_t(x)).
$$
\end{lemma}



The next lemma will be used later. 
\begin{lemma}\label{finiteness}
Let $q<0$ and  $\varphi \in \mathcal{L}_o'$. If $0<\int_{\mathbb{R}^n}\varphi^{2-q}(x)d\gamma_n(x)<+\infty$, then
$$
\int_{\mathbb{R}^n}\varphi^*(\nabla \varphi(x))\varphi^{-q}(x)d\gamma_n(x)<+\infty.
$$
\end{lemma}

\begin{proof}
 By a direct calculation, we have
\begin{eqnarray*}\label{}
&&\int_{\mathbb{R}^n}\varphi^*(\nabla \varphi(x))\varphi^{-q}(x)d\gamma_n(x)\\
&&\quad =\int_{\mathbb{R}^n}(\langle x,\nabla \varphi(x)\rangle- \varphi(x))\varphi^{-q}(x)d\gamma_n(x)\\
&&\quad =\frac{1}{1-q}\int_{\mathbb{R}^n}\langle x,\nabla \varphi^{1-q}(x)\rangle d\gamma_n(x)-\int_{\mathbb{R}^n}\varphi^{1-q}(x)d\gamma_n(x)\\
&&\quad =\frac{(2\pi)^{-\frac{n}{2}}}{1-q}\int_{\mathbb{R}^n}  {\rm div}(\varphi^{1-q}e^{-\frac{|x|^2}{2}}x) dx+\frac{1}{1-q}\int_{\mathbb{R}^n}|x|^2
\varphi(x)^{1-q}d\gamma_n(x)\\
&&\quad\quad\quad -\frac{n+q-1}{1-q}\int_{\mathbb{R}^n}
\varphi(x)^{1-q}d\gamma_n(x).
\end{eqnarray*}
Since there exist constants  $a>0$, $b>0$  such that  $\varphi(x)\leq b |x|^{1+a} $  when $|x|$ is big enough  
 for $x\in \mathbb{R}^n$, hence
\begin{eqnarray*}\label{}
&&\int_{\mathbb{R}^n}  {\rm div}(\varphi^{1-q}e^{-\frac{|x|^2}{2}}x) dx\\
&&\quad\quad=\lim_{r\rightarrow+\infty}\int_{rB}  {\rm div}(\varphi^{1-q}e^{-\frac{|x|^2}{2}}x) dx\\
&&\quad\quad=\lim_{r\rightarrow+\infty}re^{-\frac{r^2}{2}}\int_{\partial(rB)}  \varphi^{1-q} (x)d\mathcal{H}^{n-1}(x)\\
&&\quad\quad=0.
\end{eqnarray*}
The H\"older inequality yields
\begin{eqnarray*}\label{}
\int_{\mathbb{R}^n}|x|^2
\varphi(x)^{1-q}d\gamma_n(x)\leq\left[\int_{\mathbb{R}^n}
\varphi(x)^{2-q}d\gamma_n(x)\right]^{\frac{1-q}{2-q}}\left[\int_{\mathbb{R}^n}
|x|^{2(2-q)}d\gamma_n(x)\right]^{\frac{1}{2-q}}.
\end{eqnarray*}
The finiteness of $\int_{\mathbb{R}^n}\varphi^*(\nabla \varphi(x))\varphi^{-q}(x)d\gamma_n(x)$ follows from the above facts, $\widetilde{W}_{n+1-q}(\varphi)$ is finite and $\int_{\mathbb{R}^n}
|x|^{2(2-q)}d\gamma_n(x)$ is a constant depend on $n$ and $q$.
\end{proof}

A convex function $\varphi$ is an \emph{admissible perturbation} for the convex function  $\psi$ if there exists a constant $c>0$ such that the function
\begin{eqnarray}\label{admissible perturbation}
\varphi^*-c\psi^*
\end{eqnarray}
is  convex.

The following theorem  provides an integral formula of   the  $(n+1-q)$-th functional dual mixed  quermassintegral.

\begin{thm}\label{integral formula}
Let $q\leq0$ and $\varphi,\psi \in \mathcal{L}_o'$ and assume that $\varphi$ is an admissible perturbation for  $\psi$. If
 $\varphi(0)=\psi(0)=0$ and $\widetilde{W}_{n+2-q}(\varphi)$ is finite, then $\widetilde{W}_{n+1-q}(\varphi,\psi)\in[0,+\infty]$ and
\begin{eqnarray}\label{integral formula1}
\widetilde{W}_{n+1-q}(\varphi,\psi)
&=&\int_{\mathbb{R}^n}\psi^*(\nabla\varphi(y))\varphi(y)^{-q}d\gamma_n(y)\nonumber\\
&=:&\int_{\mathbb{R}^n}\psi^*(x)d\widetilde{C}_q(\varphi,x).
\end{eqnarray}
\end{thm}
\begin{proof}

The proof is divided as  several  steps.

\vskip 0.3cm
\emph{Step 1: $\widetilde{W}_{n+1-q}(\varphi,\psi)\in[0,+\infty]$.}
\vskip 0.3cm

Let $\varphi,\psi$ be  non-negative convex functions in $\mathbb{R}^n$ and, for any $t>0$, set $\varphi_t=\varphi\Box(\psi t)$. Since $\psi(0)=0$,  then  Lemma \ref{Colesanti's lemma 4.10} (1) yields that the function $\varphi_t$ is pointwise decreasing with respect to $t$.

 Because of Lemma \ref{Colesanti's lemma 4.10} (1) and (2), so that for every $x\in \mathbb{R}^n$ there exists $\bar{\varphi}(x)=\lim_{t\rightarrow 0^+} \varphi_t(x)$ and it holds $\bar{\varphi}(x)\leq \varphi(x)$. By the monotone convergence theorem  and $q\leq 0$, we have
$$
\widetilde{W}_{n+1-q}(\bar{\varphi})=\lim_{t\rightarrow 0^+} \widetilde{W}_{n+1-q}(\varphi_t)\leq \widetilde{W}_{n+1-q}(\varphi).
$$
Let us consider separately the two cases $\widetilde{W}_{n+1-q}(\bar{\varphi})<\widetilde{W}_{n+1-q}(\varphi)$ and $\widetilde{W}_{n+1-q}(\bar{\varphi})= \widetilde{W}_{n+1-q}(\varphi)$.

If $\widetilde{W}_{n+1-q}(\bar{\varphi})< \widetilde{W}_{n+1-q}(\varphi)$, then
\begin{eqnarray*}\label{}
\lim_{t\rightarrow0^{+}}\frac{\widetilde{W}_{n+1-q}(\varphi_t)-\widetilde{W}_{n+1-q}(\varphi)}{t}=-\infty.
\end{eqnarray*}
If $\widetilde{W}_{n+1-q}(\bar{\varphi})= \widetilde{W}_{n+1-q}(\varphi)$, we further distinguish the following two subcases:
\begin{eqnarray*}\label{}
\exists \quad t_0>0\quad \text{such that}\quad \widetilde{W}_{n+1-q}(\varphi_{t_0})= \widetilde{W}_{n+1-q}(\varphi),
\end{eqnarray*}
or
\begin{eqnarray*}\label{}
 \widetilde{W}_{n+1-q}(\varphi_{t})< \widetilde{W}_{n+1-q}(\varphi)\quad \text{for all}\quad t>0.
\end{eqnarray*}
Lemma \ref{Colesanti's lemma 4.10} (1) declares  that $\widetilde{W}_{n+1-q}(\varphi_t)$ is a monotone decreasing function of $t$, so that $\widetilde{W}_{n+1-q}(\varphi_{t})= \widetilde{W}_{n+1-q}(\varphi)$ for every $t\in[0,t_0]$  in the former subcase. This means
\begin{eqnarray*}\label{}
\lim_{t\rightarrow0^{+}}\frac{\widetilde{W}_{n+1-q}(\varphi_t)-\widetilde{W}_{n+1-q}(\varphi)}{t}=0.
\end{eqnarray*}
Now, we consider the latter subcase.  Since $\widetilde{W}_{n+1-q}(\varphi_t)$ is a decreasing convex function of $t$ when $q\leq0$, hence
\begin{eqnarray}\label{3.2}
\lim_{t\rightarrow0^{+}}\frac{\widetilde{W}_{n+1-q}(\varphi_t)-\widetilde{W}_{n+1-q}(\varphi)}{t}\in[-\infty,0].
\end{eqnarray}
Combining the above facts and the definition of functional dual quermassintegrals,  we conclude that $\widetilde{W}_{n+1-q}(\varphi,\psi)\in[0,+\infty]$.

\vskip 0.3cm

\emph{Step 2: For $t>0$, it holds
\begin{eqnarray}\label{step 2-1}
\widetilde{W}_{n+1-q}(\varphi_t)-\widetilde{W}_{n+1-q}(\varphi)=\int_0^t\Phi(s)ds,
\end{eqnarray}
where
\begin{eqnarray}\label{step 2-2}
\Phi(s)=\int_{\mathbb{R}^n}\psi^*(x)d\widetilde{C}_q(\varphi_s,x).
\end{eqnarray}
}
\vskip 0.3cm

 By Lemma \ref{Colesanti's lemma 4.11}, it holds
\begin{eqnarray}\label{step 2-3}
\lim_{h\rightarrow0}\frac{\varphi_{t+h}^{1-q}(x)-\varphi_{t}^{1-q}(x)}{h}
=(q-1)\psi^*(\nabla\varphi_{t}(x))\varphi_{t}^{-q}(x),
\end{eqnarray}
for $x\in \mathbb{R}^n$. From the facts $q\leq0$ and  $\psi(0)=0$,  Lemma \ref{Colesanti's lemma 4.10} (1) and (3), we infer that for every $s\in[0,1]$, the non-negative function $\psi^*(\nabla\varphi_{s}(x))\varphi_{s}^{-q}(x)$ are bounded above by some continuous function independent of $s$. Then, by the pointwise convergence in (\ref{step 2-3}) and dominated convergence theorem, we have
\begin{eqnarray*}
&&\frac{1}{q-1}\lim_{h\rightarrow0}
\frac{\widetilde{W}_{n+1-q}(\varphi_{t+h})-\widetilde{W}_{n+1-q}(\varphi_t)}{h}\\
&&\quad =\frac{1}{q-1}\lim_{h\rightarrow0}
\int_{\mathbb{R}^n}\frac{\varphi_{t+h}^{1-q}(x)-\varphi_{t}^{1-q}(x)}{h}d\gamma_n(x)\\
&&\quad =\int_{\mathbb{R}^n}\psi^*(\nabla\varphi_{t}(x))\varphi_{t}^{-q}(x)d\gamma_n(x).
\end{eqnarray*}
Moreover,
\begin{eqnarray*}\label{}
\widetilde{W}_{n+1-q}(\varphi_t)-\widetilde{W}_{n+1-q}(\varphi)
=\int_0^t\int_{\mathbb{R}^n}\psi^*(x)d\widetilde{C}_q(\varphi_s,x)ds.
\end{eqnarray*}

\vskip 0.3cm

\emph{Step 3: The function $\Phi$ defined in (\ref{step 2-2}) takes finite values at every $s>0$.}
\vskip 0.3cm

Let $s>0$. Since $\varphi(0)=0$, hence $\varphi^*\geq0$ and
\begin{eqnarray*}\label{}
s\Phi(s)&\leq&\int_{\mathbb{R}^n}(\varphi^*(x)+s\psi^*(x))d\widetilde{C}_q(\varphi_s,x)\\
&=&\int_{\mathbb{R}^n}\varphi_s^*(x)d\widetilde{C}_q(\varphi_s,x).
\end{eqnarray*}
 Lemma \ref{finiteness} ensures $\Phi(s)<+\infty$ when $s>0$.

 Let now $s=0$. Since $\psi$ is an admissible perturbation for $\varphi$, by its definition (\ref{admissible perturbation}) and Lemma \ref{Colesantis lemma2.5},  there exist constants $a,c>0$ and $b\in \mathbb{R}$ such that
\begin{eqnarray}\label{bounded}
(\varphi^*-c\psi^*)(y)\geq a|y|+b,
\end{eqnarray}
for $y\in \mathbb{R}^n$. Hence,
\begin{eqnarray*}\label{}
\Phi(0)&=&\int_{\mathbb{R}^n}\psi^*(x)d\widetilde{C}_q(\varphi,x)\\
&\leq& c^{-1}\int_{\mathbb{R}^n}\varphi^*(x)d\widetilde{C}_q(\varphi,x)-\frac{b}{c}\widetilde{W}_{n-q}(\varphi)
-\frac{a}{c}\int_{\mathbb{R}^n}|\nabla\varphi(x)|\varphi^{-q}(x)d\gamma_n(x)\\
&\leq& c^{-1}\int_{\mathbb{R}^n}\varphi^*(x)d\widetilde{C}_q(\varphi,x)-\frac{b}{c}\widetilde{W}_{n-q}(\varphi).
\end{eqnarray*}
The finiteness of $\Phi(0)$  follows from Lemma \ref{finiteness}, and $a>0,c>0$.

\vskip 0.3cm

\emph{Step 4: The function $\Phi$ defined in (\ref{step 2-2}) is continuous at every $s>0$, and it is continuous from the right at $s=0$.}
\vskip 0.3cm

 Let $y=\nabla\varphi_s(x)$. Then we have 
\begin{eqnarray*}\label{}
\Phi(s)=(2\pi)^{-\frac{n}{2}}\int_{\mathbb{R}^n}\psi^*(y)\varphi_s(\nabla\varphi_s^*(y))e^{-\frac{|\nabla\varphi_s^*(y)|^2}{2}}\det(\nabla^2\varphi_s^*(y))dy.
\end{eqnarray*}
 By the convexity of $\phi=\varphi^*-c\psi^*$  for some $c>0$,  one gets 
\begin{eqnarray*}\label{}
\nabla^2\phi=\nabla^2\varphi^*-c\nabla^2\psi^*, 
\end{eqnarray*}
and it is a symmetric positive semi-definite matrix. This implies that  
\begin{eqnarray*}\label{}
\nabla^2\psi^*=c^{-1}\nabla^2\varphi^*-c^{-1} \nabla^2\phi.
\end{eqnarray*}
According to the fact that 
\begin{eqnarray*}\label{}
\det(A+B)^{1/n}\geq\det(A)^{1/n}+\det(B)^{1/n},
\end{eqnarray*}
for symmetric positive semi-definite matrixes $A$ and $B$, we infer that 
\begin{eqnarray}\label{1-0}
\det(\nabla^2\varphi_s^*(y))&=&\det(\nabla^2\varphi^*(y)+s\nabla^2\psi^*(y))\nonumber \\
&=&\det((1+c^{-1}s)\nabla^2\varphi^*(y)-c^{-1}s \nabla^2\phi)\nonumber \\
&\leq&(1+c^{-1}s)^n\det(\nabla^2\varphi^*(y)).
\end{eqnarray}
Obviously, 
\begin{eqnarray*}\label{}
|\nabla \varphi_{s}^*(y)|^2=\big| \nabla \varphi_{}^*(y)+s \nabla \psi_{}^*(y)\big|^2
\end{eqnarray*}
is a continuous function with respect to $s$. Hence, 
$$\lim_{s\to 0}|\nabla \varphi_{s}^*(y)|^2=\lim_{s\to 0}|\nabla \varphi_{}^*(y)|^2, $$
i.e.,  for any $0<\varepsilon<1$ there exists an $s_0>0$ small enough such that  
\begin{eqnarray}\label{2-0}-\varepsilon<|\nabla \varphi_{s}^*(y)|^2-|\nabla \varphi_{}^*(y)|^2<\varepsilon
\end{eqnarray} 
when $0<s<s_0$. 
On the other hand,  by the convexity of $\phi=\varphi^*-c\psi^*$  for some $c>0$ and $\varphi(0)=\psi(0)=0$, we have
\begin{eqnarray}\label{3}
 \varphi_s(\nabla\varphi_s^*(y)) &=&\langle y,\nabla\varphi_s^*(y)\rangle-\varphi_s^*(y)\nonumber\\
 &=&(1+\tfrac{s}{c})\Big(\langle y,\nabla\varphi^*(y)\rangle-\varphi^*(y)\Big)-\tfrac{s}{c}\Big(\langle y, \nabla\phi(y)\rangle-\phi(y)\Big)\nonumber\\
&=&(1+\tfrac{s}{c})\varphi(\nabla\varphi^*(y))-\tfrac{s}{c} \phi^*( \nabla\phi(y) ) \nonumber\\
&\leq&(1+\tfrac{1}{c})\varphi(\nabla\varphi^*(y)).
\end{eqnarray}

 From \eqref{1-0}, \eqref{2-0}, \eqref{3}  and \eqref{bounded} for $s\in(0,\min\{s_0,1\})$  we have
\begin{eqnarray*}\label{}
&&\psi^*(y)\varphi_s(\nabla\varphi_s^*(y))^{-q}\det(\nabla^2\varphi_s^*(y))e^{-\frac{|\nabla\varphi_s^*(y)|^2}{2}}\\
 &&\quad  \leq(1+c^{-1})^{n-q}e\psi^*(y)\varphi(\nabla\varphi^*(y))^{-q}\det(\nabla^2\varphi^*(y))e^{-\frac{|\nabla\varphi^*(y)|^2}{2}}\\
&&\quad  \leq(1+c^{-1})^{n-q}c^{-1}e\psi^*(y)\varphi(\nabla\varphi^*(y))^{-q}\det(\nabla^2\varphi^*(y))e^{-\frac{|\nabla\varphi^*(y)|^2}{2}} \\
&&\quad \quad-\frac{b}{c} \varphi(\nabla\varphi^*(y))^{-q}\det(\nabla^2\varphi^*(y))e^{-\frac{|\nabla\varphi^*(y)|^2}{2}}.
\end{eqnarray*}
Let $h_1(y)=(2\pi)^{-\frac{n}{2}}(1+c^{-1})^nc^{-1}e\varphi^*(y)\varphi(\nabla\varphi^*(y))^{-q}\det(\nabla^2\varphi^*(y))e^{-\frac{|\nabla\varphi^*(y)|^2}{2}}$ and $h_2(y)=\max\{0,-\frac{b}{c} \}(2\pi)^{-\frac{n}{2}}\varphi(\nabla\varphi^*(y))^{-q}\det(\nabla^2\varphi^*(y))e^{-\frac{|\nabla\varphi^*(y)|^2}{2}}$. Then
\begin{eqnarray*}
\int_{\mathbb{R}^n}h_1(y)dy&=&(2\pi)^{-\frac{n}{2}}(1+c^{-1})^nc^{-1}e\int_{\mathbb{R}^n}\varphi^*(y)\varphi(\nabla\varphi^*(y))^{-q}\det(\nabla^2\varphi^*(y))e^{-\frac{|\nabla\varphi^*(y)|^2}{2}}dy\\
&=&(1+c^{-1})^nc^{-1}\int_{\mathbb{R}^n}\varphi^*(\nabla\varphi(x))\varphi(x)^{-q}d\gamma_n(x),
\end{eqnarray*}
and 
\begin{eqnarray*}
\int_{\mathbb{R}^n}h_2(y)dy=\max\{0,-\tfrac{b}{c} \}\widetilde{W}_{n+1-q}(\varphi).
\end{eqnarray*}
From Lemma \ref{finiteness}, we obtain  the integrability of the non-negative  function $h_1+h_2$
with respect to the Gaussian probability measure $\gamma_n$. Therefore,  the desired claim follows from dominated convergence theorem.

\vskip 0.3cm

\emph{Step 5: Equality (\ref{integral formula1}) holds.}
\vskip 0.3cm

Combining with (\ref{step 2-1}), the finiteness (which proved in Step 3) and continuity  (which proved in Step 4) of $\Phi(s)$ for $s>0$,  we have
\begin{eqnarray}\label{step5-1}
\Phi(s)=\frac{d}{dt}\widetilde{W}_{n+1-q}(\varphi_t)\Big|_{t=s}.
\end{eqnarray}
The continuity from the right of $\Phi$ at $s=0$  implies
\begin{eqnarray}\label{step5-2}
\lim_{s\rightarrow0^{+}}\Phi(s)=\Phi(0)=\int_{\mathbb{R}^n}\psi^*(x)d\widetilde{C}_q(\varphi,x).
\end{eqnarray}
Then, equality (\ref{integral formula1}) follows from (\ref{step 2-1}), (\ref{step5-1}) and (\ref{step5-2}).
\end{proof}


\vskip 0.3cm
~
\subsection{Minkowski version inequality}
~
\vskip 0.3cm

There is a Minkowski inequality for the  functional  dual mixed  quermassintegral.

\begin{lemma}
Let $\varphi,\psi \in \mathcal{L}$.   For any  $t\in (0,1)$ and $x,y\in \mathbb{R}^n$, then
\begin{eqnarray}\label{inequality for inf1}
&&[\varphi(1-t)]\square[\psi t]((1-t)x+ty) \leq(1-t)\varphi(x)+t\psi(y),
\end{eqnarray}
with equality if and only if $\varphi(x)=\psi(x-x_0)$ for some $x_0\in \mathbb{R}^n$.
\end{lemma}
\begin{proof}
For any $x,y\in \mathbb{R}^n$ and $t\in (0,1)$, we have
\begin{eqnarray*}\label{}
&&[\varphi(1-t)]\square[\psi t]((1-t)x+ty) \nonumber  \\
&& \quad =\inf_{z\in\mathbb{R}^n}\left\{(1-t)\varphi\left(\frac{(1-t)x+ty-z}{1-t}\right)+t\psi\left(\frac{z}{t}\right)\right\} \nonumber\\
&&\quad\leq(1-t)\varphi(x)+t\psi(y),
\end{eqnarray*}
where in the last inequality we have used  $z=ty$. This is the desired inequality.

The inequality (\ref{inequality for inf1}) is equivalent to
\begin{eqnarray*}\label{}
e^{-[\varphi(1-t)]\square[\psi t]((1-t)x+ty)}  \geq\left(e^{-\varphi(x)}\right)^{1-t}\left(e^{-\psi(y)}\right)^t.
\end{eqnarray*}
Then, the Pr\'ekopa-Leindler inequality (\ref{Prekopa-Leindler inequality}) deduces
\begin{eqnarray}\label{inequality for inf3}
\int_{\mathbb{R}^n}e^{-[\varphi(1-t)]\square[\psi t](x)} dx \geq\left(\int_{\mathbb{R}^n}e^{-\varphi(x)}dx\right)^{1-t}\left(\int_{\mathbb{R}^n}e^{-\psi(y)}dx\right)^t.
\end{eqnarray}
If there is an equality  in (\ref{inequality for inf1}), then H\"older inequality yields
\begin{eqnarray*}
\int_{\mathbb{R}^n}e^{-[\varphi(1-t)]\square[\psi t](x)} dx &=&\int_{\mathbb{R}^n}\left(e^{-\varphi(x)}\right)^{1-t}\left(e^{-\psi(x)}\right)^{t}dx. \\
&\leq&\left(\int_{\mathbb{R}^n}e^{-\varphi(x)}dx\right)^{1-t}\left(\int_{\mathbb{R}^n}e^{-\psi(y)}dx\right)^t.
\end{eqnarray*}
This means that  there is an equality in  the Pr\'ekopa-Leindler inequality  (\ref{inequality for inf3}). The equality condition of Pr\'ekopa-Leindler inequality guarantees that  convex functions $\varphi$ and $\psi$ agree up to a translation. Conversely, it is not hard to check that there is an equality in (\ref{inequality for inf1}) if $\varphi(x)=\psi(x-x_0)$ for some $x_0\in \mathbb{R}^n$.
\end{proof}

We are now in the position to obtain the following Brunn-Minkowski inequality for the functional dual quermassintegrals. 

\begin{thm}\label{functional B-M ineq}
Let $\varphi,\psi \in \mathcal{L}_o$. If  $\overline{W}_{n-q}(\varphi)$ and $\overline{W}_{n-q}(\psi )$ are finite and $q\leq-1$,   for any $t\in (0,1)$,  then
\begin{eqnarray}\label{B-Minkowski inequality}
\overline{W}_{n-q}([\varphi(1-t)]\square[\psi t])\leq (1-t)\overline{W}_{n-q}(\varphi)+ t\overline{W}_{n-q}(\psi )
\end{eqnarray}
with equality  if and only if $\varphi(x)=\psi(x-x_0)$ with $x_0\in \mathbb{R}^n$ when $q=-1$ and $\varphi=\psi$ when $q<-1$.

%
%
%
\end{thm}
\begin{proof}
Since $q\leq-1$, (\ref{inequality for inf1}) and Minkowski's inequality, we have
\begin{eqnarray*}\label{}
\overline{W}_{n-q}([\varphi(1-t)]\square[\psi t])&\leq&\left(\int_{\mathbb{R}^n}((1-t)\varphi(x)+t\psi(y))^{-q}d\gamma_n(x)\right)^{\frac{1}{-q}} \\
 &\leq&(1-t)\overline{W}_{n-q}(\varphi)+ t\overline{W}_{n-q}(\psi ).
\end{eqnarray*}
The equality conditions follow from the equality conditions of inequality (\ref{inequality for inf1}) and Minkowski's inequality.
\end{proof}

Similar to the proof of Theorem \ref{functional B-M ineq}, we also obtain the case of $q>-1$ and $q\neq0$.
\begin{cor}
Let $\varphi,\psi \in \mathcal{L}_o$. If $\overline{W}_{n-q}(\varphi)$ and $\overline{W}_{n-q}(\psi )$ are finite and $q>-1$ and $q\neq0$,  then
\begin{eqnarray}\label{B-Minkowski inequality}
\overline{W}_{n-q}([\varphi(1-t)]\square[\psi t])\geq (1-t)\overline{W}_{n-q}(\varphi)+ t\overline{W}_{n-q}(\psi )
\end{eqnarray}
with equality  if and only if  $\varphi=\psi$.

\end{cor}

The following lemma is useful.

\begin{lemma}\label{part proof of Minkowski inequality}
Let $q\leq0$ and $\varphi,\psi \in \mathcal{L}_o$. Then
\begin{eqnarray*}
&&\lim_{t\rightarrow0^{+}}\frac{\widetilde{W}_{n+1-q}\big([\varphi(1-t)]\square[\psi t]\big)-\widetilde{W}_{n+1-q}(\varphi)}{t}\\
&&\quad\quad=(q-1)\widetilde{W}_{n+1-q}(\varphi,\psi)-(q-1)\widetilde{W}_{n+1-q}(\varphi,\varphi).
\end{eqnarray*}
\end{lemma}
\begin{proof}
Let $a(t)=\frac{t}{1-t}$, for $t\in(0,1)$. From the definition of infimal convolution (\ref{infimal convolution}), we have
\begin{eqnarray*}\label{}
&&\big[\varphi\Box [\psi a(t)]\big](1-t)(x)\\
&&\quad\quad=(1-t)\big[\varphi\Box [\psi a(t)]\big]\left(\frac{x}{1-t}\right)\\
&&\quad\quad=(1-t)\inf_{y\in\mathbb{R}^n}\left\{\varphi\left(\frac{x-(1-t)y}{1-t}\right)+\frac{t}{1-t}\psi\left(\frac{(1-t)y}{t}\right)\right\}\\
&&\quad\quad=\inf_{y\in\mathbb{R}^n}\left\{(1-t)\varphi\left(\frac{x-y}{1-t}\right)+t\psi\left(\frac{y}{t}\right)\right\}\\
&&\quad\quad=[\varphi(1-t)]\square[\psi t] (x),
\end{eqnarray*}
and 
\begin{eqnarray}\label{part proof of Minkowski inequality1}
&&\frac{\widetilde{W}_{n+1-q}\big([\varphi(1-t)]\square[\psi t]\big)-\widetilde{W}_{n+1-q}(\varphi)}{t}\nonumber \\
&&\quad\quad =\frac{\widetilde{W}_{n+1-q}\big([\varphi\Box [\psi a(t)]\big](1-t)\big)-\widetilde{W}_{n+1-q}(\varphi\Box [\psi a(t))}{t}\nonumber\\
&&\quad\quad\quad\quad\quad\quad+\frac{\widetilde{W}_{n+1-q}(\varphi\Box [\psi a(t))-\widetilde{W}_{n+1-q}(\varphi)}{t}.
\end{eqnarray}
We set $\varphi_{a(t)}=\varphi\Box [\psi a(t)]$, and
$$
b_t(s)=\widetilde{W}_{n+1-q}\big(\varphi_{a(t)}(1-s)\big),
$$
for $s\in(0,1)$. By  a direct calculation,
\begin{eqnarray*}
\frac{d}{ds}b_t(s)&=&(2\pi)^{-\frac{n}{2}}\frac{d}{ds}(1-s)^{1-q}\int_{\mathbb{R}^n}
\left(\varphi_{a(t)}\left(\frac{x}{1-s}\right)\right)^{1-q}e^{-\frac{|x|^2}{2}}dx\\
&=&(2\pi)^{-\frac{n}{2}}\frac{d}{ds}(1-s)^{n+1-q}\int_{\mathbb{R}^n}
\left(\varphi_{a(t)}\left(x\right)\right)^{1-q}e^{-\frac{(1-s)^2|x|^2}{2}}dx\\
&=&-(n+1-q)(2\pi)^{-\frac{n}{2}}(1-s)^{n-q}\int_{\mathbb{R}^n}
\left(\varphi_{a(t)}\left(x\right)\right)^{1-q}e^{-\frac{(1-s)^2|x|^2}{2}}dx\\
&&\quad +\ (2\pi)^{-\frac{n}{2}}(1-s)^{n+1-q}\int_{\mathbb{R}^n}
\left(\varphi_{a(t)}\left(x\right)\right)^{1-q}e^{-\frac{(1-s)^2|x|^2}{2}}(1-s)|x|^2dx.
\end{eqnarray*}
Then for every fixed $t\in(0,1)$, by Lagrange theorem we have that there exists a $\bar{s}\in(0,1)$ such that
\begin{eqnarray*}
&&\frac{\widetilde{W}_{n+1-q}\big([\varphi\Box [\psi a(t)]\big](1-t)\big)-\widetilde{W}_{n+1-q}(\varphi\Box [\psi a(t))]}{t}\\
&&\quad\quad=b'_t(\bar{s})\\
&&\quad\quad=-(n+1-q)(2\pi)^{-\frac{n}{2}}(1-\bar{s})^{n-q}\int_{\mathbb{R}^n}
\left(\varphi_{a(t)}\left(x\right)\right)^{1-q}e^{-\frac{(1-\bar{s})^2|x|^2}{2}}dx\\
&&\quad\quad\quad +\ (2\pi)^{-\frac{n}{2}}(1-\bar{s})^{n+1-q}\int_{\mathbb{R}^n}
\left(\varphi_{a(t)}\left(x\right)\right)^{1-q}e^{-\frac{(1-\bar{s})^2|x|^2}{2}}(1-\bar{s})|x|^2dx.
\end{eqnarray*}

By Lemma \ref{Colesanti's lemma 4.10}, as $t\rightarrow 0^{+}$  functions $\varphi_{a(t)}$ converge decreasing to a convex  function $\tilde{\varphi}$. Then, by monotone convergence theorem, $\bar{s}\rightarrow 0^{+}$ and $a(t)\rightarrow 0^{+}$ as $t\rightarrow 0^{+}$, we obtain
\begin{eqnarray}\label{part proof of Minkowski inequality2}
&&\lim_{t\rightarrow0^{+}}\frac{\widetilde{W}_{n+1-q}\big([\varphi\Box [\psi a(t)]\big](1-t)\big)-\widetilde{W}_{n+1-q}(\varphi\Box [\psi a(t)])}{t}\nonumber \\
&&\quad\quad=-(n+1-q)\widetilde{W}_{n+1-q}(\tilde{\varphi})+ \int_{\mathbb{R}^n}
\tilde{\varphi}\left(x\right)^{1-q}|x|^2d\gamma_n(x).
\end{eqnarray}
Concerning the second addendum in (\ref{part proof of Minkowski inequality1}),  Definition \ref{definion of mixed functional dual  quermassintegrals} shows immediately that
\begin{eqnarray}\label{part proof of Minkowski inequality3}
&&\lim_{t\rightarrow0^{+}}\frac{\widetilde{W}_{n+1-q}(\varphi\Box [\psi a(t))-\widetilde{W}_{n+1-q}(\varphi)}{t}=(q-1)\widetilde{W}_{n+1-q}(\varphi,\psi).
\end{eqnarray}

Recall that the functions $\varphi_{a(t)}$ converge decreasing to some convex positive function $\tilde{\varphi}$ as $t\rightarrow0^+$. In order to obtain the desired equality, we may distinguish the two cases of $\widetilde{W}_{n+1-q}(\tilde{\varphi})\leq\widetilde{W}_{n+1-q}(\varphi)$ and $\widetilde{W}_{n+1-q}(\tilde{\varphi})=\widetilde{W}_{n+1-q}(\varphi)$. If $\widetilde{W}_{n+1-q}(\tilde{\varphi})\leq\widetilde{W}_{n+1-q}(\varphi)$, the limit in
(\ref{part proof of Minkowski inequality3}) becomes $-\infty$, and the limit in (\ref{part proof of Minkowski inequality2}) is finite, hence
\begin{eqnarray*}
\lim_{t\rightarrow0^{+}}\frac{\widetilde{W}_{n+1-q}\big([\varphi(1-t)]\square[\psi t]\big)-\widetilde{W}_{n+1-q}(\varphi)}{t}=-\infty,
\end{eqnarray*}
and the result of the lemma holds true.  If $\widetilde{W}_{n+1-q}(\tilde{\varphi})=\widetilde{W}_{n+1-q}(\varphi)$, then $\tilde{\varphi}=\varphi$ almost everywhere, so that the right hand side of (\ref{part proof of Minkowski inequality2}) is $(1-q)\widetilde{W}_{n+1-q}(\varphi,\varphi)$. The lemma follows by summing up (\ref{part proof of Minkowski inequality2}) and (\ref{part proof of Minkowski inequality3}).
\end{proof}

By Lemma \ref{part proof of Minkowski inequality}, we obtain the following Minkowski inequality for the functional dual mixed  quermassintegrals.

\begin{thm}\label{Minkowski inequality}
Let $\varphi,\psi \in \mathcal{L}_o$. If $q\leq0$, then
\begin{eqnarray}\label{Minkowski inequality}
\widetilde{W}_{n+1-q}(\varphi,\psi)\geq\frac{\overline{W}_{n+1-q}(\psi)^{}
 -\overline{W}_{n+1-q}(\varphi)^{}}{\overline{W}_{n+1-q}(\varphi)^{q}}
 +\widetilde{W}_{n+1-q}(\varphi,\varphi),
\end{eqnarray}
with equality  if and only if $\varphi(x)=\psi(x-x_0)$ with $x_0\in \mathbb{R}^n$ when $q=0$ and $\varphi(x)=\psi(x)$ when $q<0$.
\end{thm}

\begin{proof}

Let
$$
\Upsilon(t)=\widetilde{W}_{n+1-q}([\varphi(1-t)]\square[\psi t])^{\frac{1}{1-q}}.
$$
From the Brunn-Minkowski style inequality (\ref{B-Minkowski inequality}),  we noticed that $\Upsilon(t)$ is convex on $[0,1]$, hence
\begin{eqnarray}\label{main inequality 0}
\Upsilon(t)\leq\Upsilon(0)+t[\Upsilon(1)-\Upsilon(0)],
\end{eqnarray}
for $t\in[0,1]$. As a consequence, the derivative of the function $\Upsilon$ at $t=0$ satisfies
\begin{eqnarray}\label{main inequality}
\Upsilon'(0)\leq\Upsilon(1)-\Upsilon(0).
\end{eqnarray}
Lemma \ref{part proof of Minkowski inequality} deduces that
\begin{eqnarray*}\label{}
\Upsilon'(0)=\widetilde{W}_{n+1-q}(\varphi)^{\frac{q}{1-q}}\left[\widetilde{W}_{n+1-q}(\varphi,\varphi)-\widetilde{W}_{n+1-q}(\varphi,\psi)
\right].
\end{eqnarray*}
Inequality (\ref{main inequality}) can be rewritten as
\begin{eqnarray}\label{}
&&\widetilde{W}_{n+1-q}(\varphi)^{\frac{q}{1-q}}\left[
\widetilde{W}_{n+1-q}(\varphi,\varphi)-\widetilde{W}_{n+1-q}(\varphi,\psi)\right]\nonumber \\
&&\quad \leq\widetilde{W}_{n+1-q}(\psi)^{\frac{1}{1-q}}-\widetilde{W}_{n+1-q}(\varphi)^{\frac{1}{1-q}}.
\end{eqnarray}

Finally, assume that $\varphi(x)=\psi(x-x_0)$ with $x_0\in \mathbb{R}^n$ when $q=0$ and $\varphi(x)=\psi(x)$ when $q<0$, then (\ref{Minkowski inequality}) holds with equality sign. Conversely, assume that (\ref{Minkowski inequality}) holds with equality sign. By inspection of the above proof one sees immediately that also inequality (\ref{main inequality}), and hence inequality (\ref{main inequality 0}) and inequality (\ref{B-Minkowski inequality}), must hold with equality sign.  This entails that (\ref{B-Minkowski inequality}) holds as an equality, and therefore, $\varphi$ and $\psi$ agree up to a translation when $q=0$ and $\varphi(x)=\psi(x)$ when $q<0$.
\end{proof}

%
%
%
%

\vskip 0.3cm
~
\section{Dual curvature measures for convex functions}\label{functional dual curvature measures}

Theorem \ref{integral formula} provides a suitable way to define dual curvature measures for convex functions, and   we will investigate the functional dual curvature measure in this section.

\begin{defi}\label{definition}
Let  $q\in \mathbb{R}$ and $\varphi\in \mathcal{L}_o$. If $0<\int_{\mathbb{R}^n}\varphi(x)^{-q}d\gamma_n(x)<+\infty$, then $\widetilde{C}_q(\varphi,\cdot)$, the $q$-th dual curvature measure of $\varphi$,  is defined as
\begin{eqnarray}\label{definition0}
\int_{\mathbb{R}^n}g(z)^{}d\widetilde{C}_q(\varphi,z)
=\int_{\mathbb{R}^n}g(\nabla \varphi^{}(x))^{}\varphi(x)^{-q}d\gamma_n(x),
\end{eqnarray}
for every  bounded continuous function $g:\mathbb{R}^n\rightarrow\mathbb{R}$. Here $d\gamma_n(x)=(2\pi)^{-\frac{n}{2}}e^{-\frac{|x|^2}{2}}dx$ is the Gaussian probability measure.
\end{defi}

If $\varphi$ is a $\mathcal{C}^2$-smooth and strictly convex function,  then the map 
$$\nabla\varphi:\{x\in \mathbb{R}^n:\varphi(x)<+\infty\}\rightarrow \{x\in \mathbb{R}^n:\varphi^*(x)<+\infty\}
$$
 is smooth and bijective. By the formulas (see, e.g.,  \cite{Mccann}),
\begin{eqnarray}\label{Mccann's formulas}
 x=\nabla\varphi^*(\nabla\varphi(x))\quad \text{and}\quad \nabla^2\varphi^*(\nabla\varphi(x))=(\nabla^2\varphi(x))^{-1}
\end{eqnarray}
for the points of $\{x\in \mathbb{R}^n:\varphi(x)<+\infty\}$ at which its Hessian $\nabla^2\varphi$ exists and is invertible,  we have
\begin{eqnarray*}
\int_{\Omega}g(x)^{}d\widetilde{C}_q(\varphi,x)
&=&(2\pi)^{-\frac{n}{2}}\int_{\Omega}g(\nabla\varphi(x))^{}\varphi(x)^{-q}e^{-\frac{|x|^2}{2}}dx \\
&=&(2\pi)^{-\frac{n}{2}}\int_{\Omega}g(\nabla\varphi(x))^{}\varphi(\nabla\varphi^*
(\nabla\varphi(x)))^{-q}e^{-\frac{|\nabla\varphi^*(\nabla\varphi(x))|^2}{2}}dx \\
&=&(2\pi)^{-\frac{n}{2}}\int_{\Omega}g(z)^{}\varphi(\nabla\varphi^*
(z))^{-q}e^{-\frac{|\nabla\varphi^*(z)|^2}{2}}\det(\nabla^2\varphi^*(z))dz,
\end{eqnarray*}
where $\Omega=\{x\in \mathbb{R}^n:\varphi(x)<+\infty\}\cap \{x\in \mathbb{R}^n:\varphi^*(x)<+\infty\}$. Hence, if $\varphi$ is a $\mathcal{C}^2$-smooth and strictly convex function, then $\widetilde{C}_q(\varphi,\cdot)$ is absolutely
continuous with respect to the  $n$-dimensional Hausdorff measure   and the Radon-Nikodym
derivative is
$$
\frac{d\widetilde{C}_q(\varphi,x)}{dx}=(2\pi)^{-\frac{n}{2}}\varphi(\nabla\varphi^*
(x))^{-q}e^{-\frac{|\nabla\varphi^*(x)|^2}{2}}\det(\nabla^2\varphi^*(x))\quad\text{on}\quad\Omega. 
$$
~~~

The following lemma  shows that  the Borel measure $\widetilde{C}_q(\varphi,\cdot)$ includes the Borel measure $\widetilde{C}_q(K,\cdot)$ of \cite{HuangLYZ}.

\begin{lemma}
Let $K\in \mathcal{K}_o^n$ be strictly convex  and $q<n$. Then the Borel measure $\widetilde{C}_q(\|x\|_{K},\cdot)$ includes the $q$-th dual curvature measure, $\widetilde{C}_q(K,\cdot)$, of $K$.
\end{lemma}
\begin{proof}

Let $K\in \mathcal{K}_o^n$ be strictly convex  and $\varphi(x)=\|x\|_{K}$. Let $\overline{V}_K$ denote the normalized cone measure of $K$, which is given by
$$
d\overline{V}_{K}(z)=\frac{\langle z,\nu_{K}(z)\rangle}{nV(K)}d\mathcal{H}^{n-1}(z)\quad \text{for}\quad z\in \partial K.
$$
Here $V(K)$ denotes the $n$-dimensional volume of $K$. For $x\in \mathbb{R}^n$,  we write $x=rz$, with $z\in \partial K$,  then $dx=nV(K)r^{n-1}drd\overline{V}_{K}(z)$. Since the map $x\mapsto \nabla \|x\|_{K}$ is 0-homogeneous, hence
\begin{eqnarray*}\label{}
 &&\int_{\mathbb{R}^n}g(x)^{}d\widetilde{C}_q(\varphi,x)\nonumber \\
&&\quad\quad=\int_{\mathbb{R}^n}g(\nabla \varphi^{}(x))^{}\varphi(x)^{-q}d\gamma_n(x) \nonumber\\
 &&\quad\quad= (2\pi)^{-\frac{n}{2}}  \int_{\mathbb{R}^n}g(\nabla \|x\|_K)^{}\|x\|_{K}^{-q}e^{-\frac{|x|^{2}}{2}}dx  \nonumber\\
&&\quad\quad=n(2\pi)^{-\frac{n}{2}}V(K)\int_0^{\infty}\int_{\partial K}g(\nabla \|z\|_K)^{}r^{n-q-1}e^{-\frac{r^2|z|^{2}}{2}}d\overline{V}_{K}(z)dr \nonumber\\
 &&\quad\quad=(2\pi)^{-\frac{n}{2}}2^{\tfrac{n-q-2}{2}}nV(K)\int_0^{\infty}e^{-r}r^{\frac{n-q}{2}-1}dr\int_{\partial K}g(\nabla \|z\|_K)^{}|z|^{q-n}d\overline{V}_{K}(z) \nonumber\\
 &&\quad\quad=(2\pi)^{-\frac{n}{2}}2^{\tfrac{n-q-2}{2}}\int_0^{\infty}e^{-r}r^{\frac{n-q}{2}-1}dr\int_{\partial K}g(\nabla \|z\|_K)^{}|z|^{q-n}h_K(\nu_K(z))d\mathcal{H}^{n-1}(z).
\end{eqnarray*}
Because of (see, e.g., \cite{Schneider2014}),
$$
\nabla \|z\|_{K}=\frac{\nu_{K}(z)}{\|\nu_{K}(z)\|_{K^{\circ}}},
$$
 for $z\in\partial K$, then
\begin{eqnarray*}\label{}
 &&\int_{\mathbb{R}^n}g(x)^{}d\widetilde{C}_q(\varphi,x)\nonumber \\
 &&\quad\quad=(2\pi)^{-\frac{n}{2}}2^{\tfrac{n-q-2}{2}}\int_0^{\infty}e^{-r}r^{\frac{n-q}{2}-1}dr\int_{\partial K}g\left(\tfrac{\nu_{K}(z)}{\|\nu_{K}(z)\|_{K^{\circ}}}\right)^{}|z|^{q-n}h_K(\nu_K(z))d\mathcal{H}^{n-1}(z).
\end{eqnarray*}
This means that the Borel measure $\widetilde{C}_q(\|x\|_{K},\cdot)$ takes the following form: For any Borel subsets $I\subseteq [0,\infty)$ and $\Omega\subseteq \partial K$,
\begin{eqnarray*}\label{spacial case}
 \widetilde{C}_q(\varphi,I\times \Omega)=\nu_1(I)\nu_2(\Omega).
\end{eqnarray*}
The measure $\nu_1$ is independent of the choice of $K$ and it is a measure with density
$$
(2\pi)^{-\frac{n}{2}}2^{\tfrac{n-q-2}{2}}e^{-r}r^{\frac{n-q}{2}-1}.
$$
The measure $\nu_2$ is a measure on $\partial K$ which implies the $q$-th dual curvature measure, $\widetilde{C}_q(K,\cdot)$, of $K$. In fact,  if  $\mathcal{R}(x)=\frac{x}{|x|}$ for $x\in \mathbb{R}^n\setminus\{0\}$, then
$$
\mathcal{R}_*\left(\nu_2\right)
$$
is referred to as the dual curvature measure, $\widetilde{C}_q(K,\cdot)$, of $K$, and  $
\mathcal{R}_*\left(\nu_2\right)
$ is the push-forward of $\nu_2$ via $\mathcal{R}$. This shows that the  dual curvature measure of a convex body can be recovered from the functional  dual curvature measure of a particular convex function.
\end{proof}

The next lemma will be used later.

\begin{lemma}\label{Cordero-ErausquinKlartag's lemma 3} \emph{\cite[Lemma 3]{Cordero-ErausquinKlartag2015}}
Let $\varphi: \mathbb{R}^n\rightarrow [0,+\infty]$ be an  essentially-continuous, convex function. Fix a vector $0\neq \theta \in \mathbb{R}^n$ and let $H=\theta^{\perp}\subset \mathbb{R}^n$ be the hyperplane orthogonal to $\theta$. Then, for $\mathcal{H}^{n-1}$-almost every $y\in H$, the  function
\begin{eqnarray*}\label{}
\mathbb{R}&\rightarrow &\mathbb{R}\cup\{+\infty\} \\
t&\mapsto& \varphi(y+t\theta).
\end{eqnarray*}
is continuous on $\mathbb{R}$, and locally-Lipschitz in the interior of the interval in which it is finite.
\end{lemma}

The support $\text{Supp}(\mu)$ of a measure $\mu$ in $\mathbb{R}^n$ is a closed set  that contains  all points  $x\in\mathbb{R}^n$ with $\mu(U)>0$ for any open set $U$ containing $x$.

\begin{lemma}\label{not supported in any hyperplane}
Let $\varphi: \mathbb{R}^n\rightarrow [0,+\infty]$ be  an  essentially-continuous, convex function and $q\leq0$.
If
$$
0<\int_{\mathbb{R}^n}\varphi^{1-q}(x)d\gamma_n(x)<+\infty \quad \text{and}\quad \varphi(0)=0,
$$
then the $q$-th dual curvature  measure of $\varphi$ is not supported in any hyperplane.
\end{lemma}
\begin{proof}
Assume  that $\text{Supp}(\widetilde{C}_q(\varphi,\cdot))\subseteq \theta^{\perp}$ for a unit vector $\theta\in S^{n-1}$. Without loss of generality, assume  $\theta=e_n$, where $e_n=(0,\cdots,0,1)$. For $x\in \mathbb{R}^n$, we write $x=(y,t)$ with $y\in  \mathbb{R}^{n-1}$ and $t=x_n\in  \mathbb{R}^{}$. Suppose $\varphi(x)>0$ for $x\in\mathbb{R}^n$,  then

\begin{eqnarray*}\label{}
0&=&\int_{\mathbb{R}^n}|z_n|^{}d\widetilde{C}_q(\varphi,z)\\
&=& \int_{\mathbb{R}^n}\left|\frac{\partial\varphi(x)}{\partial x_n}\right|^{}\varphi(x)^{-q}d\gamma_n(x)\\
&=&\frac{1}{1-q}\int_{\mathbb{R}^n}\left|\frac{\partial \varphi^{1-q}(x)}{\partial x_n}\right|^{}d\gamma_n(x).
\end{eqnarray*}
For almost every $(y,t)\in \mathbb{R}^n$,
\begin{eqnarray*}\label{}
\frac{\partial \varphi^{1-q}(y,t)}{\partial t}=0.
\end{eqnarray*}
Since $q\leq0$, the function $\varphi^{1-q}$ is also convex.  According to Lemma \ref{Cordero-ErausquinKlartag's lemma 3}, for almost any $y\in \mathbb{R}^{n-1}$, the function $t\rightarrow \varphi^{1-q}(y,t)$ is continuous in $\mathbb{R}$ and locally-Lipschitz in the interior of the interval $\{t: \varphi(y,t)<+\infty\}$. Therefore, for almost any $y\in\mathbb{R}^{n-1}$, the convex function $t\rightarrow \varphi^{1-q}(y,t)$  is constant in $\mathbb{R}$.  The condition $\varphi(0)=0$ implies that $\varphi\equiv0$ in $\mathbb{R}$ and so that the integral
$\int \varphi^{1-q}(y,t)e^{-\frac{t^2}{2}}dt$  equals to $0$. By Fubini theorem, the function $\varphi^{1-q}$ can not have a finite, non-zero integral. This is a contradiction.
\end{proof}

From Jensen's inequality we know that
$$
-q\mapsto\left(\int_{\mathbb{R}^n}\varphi(x)^{-q}d\gamma_n(x)\right)^{-\frac{1}{q}}
$$
is strictly monotone  increasing, unless $\varphi$ is constant on $\mathbb{R}^n$. Hence,  the condition in Lemma \ref{not supported in any hyperplane}
includes the assumption  in the Definition \ref{definition} when $q\leq0$.

\vskip 0.2cm

A function $Z$ defined on a lattice $(\mathcal{C},\vee,\wedge)$ and taking values in an abelian semi-group is called
a \emph{valuation} if
\begin{eqnarray}\label{definition of valuation}
Z(f\vee g)+Z(f\wedge g)=Z(f)+Z(g),
\end{eqnarray}
for all $f,g\in \mathcal{C}$. A function $Z$ defined on a set $\mathcal{S} \subset\mathcal{C}$ is called a valuation if (\ref{definition of valuation}) holds whenever
 $f,g,f\vee g,f\wedge g\in \mathcal{S}$.

  For $\mathcal{S}$ the space of convex bodies, $\mathcal{K}^n$, in $\mathbb{R}^n$ with $\vee$ denoting union and $\wedge$ intersection, the notion of valuation is classical. The classical valuation played a critical role in Dehn's solution of Hilbert's Third Problem
and have been a central focus in convex geometric analysis. In the 1930s, Blaschke   classified  the real-valued valuations on convex bodies that are ${\rm SL}(n)$ invariant. Blaschke's work was greatly extended by Hadwiger in his famous classification of continuous, rigid motion invariant valuations  and characterization of elementary mixed volumes.

In \cite{HuangLYZ}, authors proved that  dual curvature measures are valuations, i.e.,  $$
\widetilde{C}_q(K,\cdot)+\widetilde{C}_q(L,\cdot)=\widetilde{C}_q(K\cap L,\cdot)+\widetilde{C}_q(K\cup L,\cdot),
$$
for each $K,L\in \mathcal{K}_o^n$ such that $K\cup L\in \mathcal{K}_o^n$.

Recently, valuations were defined on function spaces. For a space $\mathcal{S}$ of real-valued
functions we denote by $f \vee g$ the pointwise maximum of $f$ and $g$ while $f \wedge g$ denotes their
pointwise minimum. The characteristic functions  $\chi_K,\chi_L$ of convex bodies  $K$  and $L$ satisfy
$$
\chi_{K\cup L}=\max\{\chi_K,\chi_L\}  \quad\text{and}\quad  \chi_{K\cap L}=\min\{\chi_K,\chi_L\},
$$
for all $K,L\in\mathcal{K}_o^n$ such that $K\cup L\in\mathcal{K}_o^n$. Valuations on function spaces can be seen as a generalization of valuations on $\mathcal{K}_o^n$.
For the classification of valuations on  function spaces one can refer  to \cite{ColesantiLombardiParapatits,ColesantiLudwigMussnig1,ColesantiLudwigMussnig2,Ludwig2011,Ludwig2012}.

We are ready to prove  that the $q$-th dual curvature measure $\widetilde{C}_q(\varphi,\cdot)$ is  a valuation.

\begin{lemma}
Let $\varphi,\psi \in \mathcal{L}_o$, and $q\in \mathbb{R}$. If $\min\{\varphi,\psi\}$ is convex, then
\begin{eqnarray}\label{valuation property}
\widetilde{C}_q(\varphi,\cdot)+\widetilde{C}_q(\psi,\cdot)=
\widetilde{C}_q(\min\{\varphi,\psi\},\cdot)
+\widetilde{C}_q(\max\{\varphi,\psi\},\cdot).
\end{eqnarray}
\end{lemma}
\begin{proof}
If $\min\{\varphi,\psi\}$ is convex, then for almost every $x\in \mathbb{R}^n$,
\begin{eqnarray*}
\nabla(\max\{\varphi,\psi\})(x)=\left\{\begin{array}{cccc}
\nabla \varphi(x)             &\text{when}&  \varphi(x)>\psi(x) \\
\nabla \psi(x)             & \text{when}&  \varphi(x)<\psi(x)\\
\nabla \varphi(x)=\nabla \psi(x) & \text{when}&  \varphi(x)=\psi(x)
\end{array} \right.,
\end{eqnarray*}
and
\begin{eqnarray*}
\nabla(\min\{\varphi,\psi\})(x)=\left\{\begin{array}{cccc}
\nabla \varphi(x)             & \text{when}& \varphi(x)<\psi(x) \\
\nabla \psi(x)             & \text{when}&  \varphi(x)>\psi(x)\\
\nabla \varphi(x)=\nabla \psi(x) & \text{when}&  \varphi(x)=\psi(x)
\end{array} \right..
\end{eqnarray*}
  Therefore,
\begin{eqnarray*}\label{}
&&\int_{\mathbb{R}^n}g(y)d\widetilde{C}_q(\max\{\varphi,\psi\},y)+ \int_{\mathbb{R}^n}g(y)d\widetilde{C}_q(\min\{\varphi,\psi\}, y)\\
&&\quad =\int_{\mathbb{R}^n}g( \nabla(\max\{\varphi,\psi\})(y))(\max\{\varphi,\psi\})^{-q}d\gamma_n(y)\\
&&\quad\quad\quad+\int_{\mathbb{R}^n}g( \nabla(\min\{\varphi,\psi\})(y))(\min\{\varphi,\psi\})^{-q}d\gamma_n(y)\\
&&\quad =\int_{\mathbb{R}^n \cap \{\varphi\geq \psi\}}g( \nabla(\max\{\varphi,\psi\})(y))(\max\{\varphi,\psi\})^{-q}d\gamma_n(y)\\
&&\quad\quad\quad +\int_{\mathbb{R}^n \cap \{\varphi< \psi\}}g( \nabla(\max\{\varphi,\psi\})(y))(\max\{\varphi,\psi\})^{-q}d\gamma_n(y)\\
&&\quad\quad\quad + \int_{\mathbb{R}^n \cap \{\varphi\geq \psi\}}g( \nabla(\min\{\varphi,\psi\})(y))(\min\{\varphi,\psi\})^{-q}d\gamma_n(y)\\
&&\quad\quad\quad+\int_{\mathbb{R}^n \cap \{\varphi< \psi\}}g( \nabla(\min\{\varphi,\psi\})(y))(\min\{\varphi,\psi\})^{-q}d\gamma_n(y)\\
&&\quad =\int_{\mathbb{R}^n \cap \{\varphi\geq \psi\}}g(\nabla\varphi(y))\varphi^{-q}d\gamma_n(y)+\int_{\mathbb{R}^n \cap \{\varphi< \psi\}}g( \nabla\psi)(y))\psi^{-q}d\gamma_n(y)\\
&&\quad\quad\quad + \int_{\mathbb{R}^n \cap \{\varphi\geq \psi\}}g( \nabla\psi(y))\psi^{-q}d\gamma_n(y)+\int_{\mathbb{R}^n \cap \{\varphi< \psi\}}g( \nabla\varphi(y))\varphi^{-q}d\gamma_n(y)\\
&&\quad =\int_{\mathbb{R}^n}g(y)d\widetilde{C}_q(\varphi,y)+\int_{\mathbb{R}^n }g( y)\psi^{-q}d\widetilde{C}_q(\psi,y),
\end{eqnarray*}
for all continuous integrable function $g$. Hence, we obtain the desired valuation property (\ref{valuation property}).
\end{proof}

\vskip 0.3cm

\section{Dual Minkowski problems for convex functions}\label{Dual Minkowski problem for convex functions}
~

\vskip 0.3cm

\subsection{A minimization problem}
~~\vskip 0.3cm

We write $L_1(\mu)$ to denotes the class of  $\mu$-integrable functions $f: \mathbb{R}^n\rightarrow  [0,\infty]$, i.e.,
\begin{eqnarray*}
L_1(\mu)=\left\{f: \int_{\mathbb{R}^n}f(x)d\mu(x)<+\infty\right\}.
\end{eqnarray*}
 Let $\mu$ be a finite Borel measure on $\mathbb{R}^n$ and $q\in \mathbb{R}$. We set
\begin{eqnarray}
\Psi_{\mu}(f)&=&\frac{1}{|\mu|}\int_{\mathbb{R}^n}f(x)d\mu(x)
-e^{-\overline{W}_{n+1-q}(f^*)},
\end{eqnarray}
for $f\in L_1(\mu)$ with $f(0)=0$.
Here $|\mu|$ denotes the total measure of $\mu$ and  $\overline{W}_{n+1-q}(f^*)$ is the normalized dual  quermassintegral, namely,
$$
\overline{W}_{n+1-q}(f^*)
=\left[\int_{\mathbb{R}^n}(f^*(x))^{1-q}d\gamma_n(x)\right]^{\frac{1}{1-q}}.
$$
We should point out  that the restriction $f(0)=0$ guarantees that $f^*(x)\geq0$, since (\ref{elementary property of Fenchel conjugate}) and the fact $f^{**}\leq f$.

We consider the minimization  problem,
\begin{eqnarray}\label{minimization  problem}
\inf_f\Big\{ \Psi_{\mu}(f):   0<\overline{W}_{n+1-q}(f^*)<\infty,f\in L_1(\mu)\ \text{and} \ f(0)=0\Big\}.
\end{eqnarray}

 If  the convex function $\varphi_0$ solves the infimum in \eqref{minimization  problem}, then the pre-given measure $\mu$  is exactly  the $q$-th dual curvature of $\varphi_0^*$. We need the following result which  has been provided in \cite{BermanBerndtsson2013} and a special case  one can found in \cite{Colesanti}. For a complete proof one can see in the recent paper \cite{Rot20}.

\begin{lemma}\label{Berman's formula}
Let $\varphi$ be a  lower semi-continuous convex function in $\mathbb{R}^n$  and let $g$ be a bounded continuous function  in $\mathbb{R}^n$. If $\varphi_t(x)=\varphi^{}(x)+tg(x)$ for sufficiently small  $t$ (i.e.,  $|t|<a$ for sufficiently small  $a$) and $x\in \mathbb{R}^n$, then
$$
\frac{d}{dt}\varphi_t^*(x)\Big |_{t=0}=-g(\nabla\varphi^*(x))
$$
at any point $x\in \mathbb{R}^n$ in which $\varphi^*$ is differentiable.
\end{lemma}

The next shows that if $\varphi_0$ solves the optimal problem \eqref{minimization  problem} then  the pre-given measure $\mu$ is the $q$-th dual curvature measure of  $\varphi^*_0$. 

\begin{lemma}\label{Existence}
Let  $\mu$ be a finite Borel measure on $\mathbb{R}^n$ that is not supported in a lower dimensional subspace with $\int_{\mathbb{R}^n}|x|d\mu(x)<\infty$ and $q\leq0$. If there exists a lower semi-continuous and convex  $\mu$-integrable function $\varphi_0: \mathbb{R}^n\rightarrow \mathbb{R}$ such that
\begin{eqnarray*}
\Psi_{\mu}(\varphi_0)=\inf_f\left\{ \Psi_{\mu}(f): f\in L_1(\mu), 0<\overline{W}_{n+1-q}(f^*)<\infty,   f(0)=0 \right\},
\end{eqnarray*}
then
$$
\mu=\tau\widetilde{C}_q(\varphi_0^*, \cdot),
$$
with $\tau=\frac{|\mu|}{\widetilde{W}_{n-q}(\varphi_0^*)}$.
\end{lemma}
\begin{proof}
 We assume that there exists a  $\varphi_0$ with $0<\overline{W}_{n+1-q}(\varphi_0^*)<\infty$ such that it is a minimizer for $\Psi_{\mu}(f)$ , i.e.,
\begin{eqnarray}\label{mini problem1}
\Psi_{\mu}(\varphi_0)=\inf_f\left\{ \Psi_{\mu}(f):   f\in L_1(\mu), 0<\overline{W}_{n+1-q}(f^*)<\infty,  f(0)=0 \right\}.
\end{eqnarray}
For any $\mu$-integrable continuous, bounded  function $g:\mathbb{R}^n\rightarrow \mathbb{R}$ with $g(0)=0$, we set $\varphi_{t}(x)=\varphi_0(x)+tg(x)^{}$.    Since $g$ is bounded, if   $|g|<M$ then  we can select a $t_0>0$ such that 
\begin{eqnarray}\label{bound}
\varphi_0-t_0M <\varphi_t<\varphi_0+t_0M ,
\end{eqnarray} and     
\begin{eqnarray*}\label{ }
\frac{1}{|\mu|}\int_{\mathbb{R}^n } \varphi_t(x)d\mu(x)
\leq\frac{1}{|\mu|}\int_{\mathbb{R}^n } \varphi_0(x)d\mu(x)+t_0M,
\end{eqnarray*}
for all $t\in[-t_0,t_0]$. 
This means that there exists $t_0>0$ such that $\varphi_t$ is $\mu$-integrable  on $\mathbb{R}^n $  and $\varphi_t(0)=0$ for all $t\in[-t_0,t_0]$. Inequalities  \eqref{bound} imply that 
\begin{eqnarray}\label{}
\varphi_0^*-t_0M <\varphi_t^*<\varphi_0^*+t_0M,
\end{eqnarray}
for all $t\in[-t_0,t_0]$.  Hence, the Minkowski inequality deduces 
\begin{eqnarray*}
 \overline{W}_{n+1-q}(\varphi_t^*) 
\leq\overline{W}_{n+1-q}(\varphi_0^*) +t_0M<+\infty. \end{eqnarray*}
Because of $\varphi_0$ is a minimizer of the optimal problem \eqref{mini problem1}, 
\begin{eqnarray*}\label{}
\Psi_{\mu}(\varphi_0)\leq\Psi_{\mu}(\varphi_t)\quad\text{for all}\quad t\in[-t_0,t_0].
\end{eqnarray*}
Hence, 
\begin{eqnarray*}\label{}
\frac{d}{dt}\Psi_{\mu}(\varphi_t)\Big|_{t=0}=0. 
\end{eqnarray*}

  By Lemma \ref{Berman's formula} deduces
$$
\frac{d}{dt}\varphi_{t}^*(x)\Big |_{t=0}=-g(\nabla\varphi_0^*(x))
$$
at any point $x\in \mathbb{R}^n$ in which $\varphi_0^*$ is differentiable.   By the dominated convergence theorem,  we have
\begin{eqnarray*}\label{}
\frac{d}{dt}\Psi_{\mu}(\varphi_t)\Big|_{t=0}&=&\frac{d}{dt}\left(\frac{1}{|\mu|}\int_{\mathbb{R}^n}\varphi_t(x)d\mu(x)\right)\Big|_{t=0}
-\frac{d}{dt}\left(e^{-\overline{W}_{n+1-q}(\varphi_t^*)}\right)\Big|_{t=0} \\
&=&\frac{1}{|\mu|}\int_{\mathbb{R}^n}g(x)d\mu(x)\\
&&-\overline{W}_{n+1-q}(\varphi_0^*)^qe^{-\overline{W}_{n+1-q}(\varphi_0^*)}\int_{\mathbb{R}^n}g(\nabla\varphi_{0}^*(x))(\varphi_{0}^*(x))^{-q}d\gamma_n(x),\end{eqnarray*}
for any   $\mu$-integrable continuous, bounded  function $g:\mathbb{R}^n\rightarrow \mathbb{R}$.  This means 
\begin{eqnarray}\label{12}
 \mu(\cdot)=
 |\mu|\overline{W}_{n+1-q}(\varphi_0^*)^qe^{-\overline{W}_{n+1-q}(\varphi_0^*)}\widetilde{C}_q(\varphi_0^*, \cdot).\end{eqnarray}
 By taking the integration from both sides of \eqref{12}, one sees that 
 \begin{eqnarray*}\label{}
\frac{1}{\widetilde{W}_{n-q}(\varphi_0^*)}=
 \overline{W}_{n+1-q}(\varphi_0^*)^qe^{-\overline{W}_{n+1-q}(\varphi_0^*)}.\end{eqnarray*}
 This finishes the proof. 
		\end{proof}

Next, we prove that the infimum in \eqref{minimization  problem} is attained.  The following lemma  can be found   in \cite{Cordero-ErausquinKlartag2015}.

\begin{lemma}\label{Cordero-ErausquinKlartag's lemma 16}
Let $\mu$ be a non-zero finite Borel measure on  $\mathbb{R}^n$ that is not supported in a lower-dimensional subspace, and let $M$ be the interior of ${\rm conv}({\rm Supp}(\mu))$. If $x_0\in M$, then there exists $C_{\mu,x_0}>0$ with the following property: For any non-negative, $\mu$-integrable, convex function $\varphi: \mathbb{R}^n\rightarrow [0,\infty]$,
$$
\varphi(x_0)\leq C_{\mu,x_0}\int_{\mathbb{R}^n}\varphi d\mu.
$$
\end{lemma}

\begin{pro}\label{existence of the minimization  problem}
Let $\mu$ be a finite Borel measure on $\mathbb{R}^n$  with $\int_{\mathbb{R}^n}|x|d\mu(x)<\infty$ that is not supported in a lower dimensional subspace. If $q\leq0$, then there exists a non-negative  convex,   lower semi-continuous and  $\mu$-integrable function $\varphi: \mathbb{R}^n\rightarrow \mathbb{R}$ with $\varphi(0)=0$ and $\overline{W}_{n+1-q}(\varphi^*)$ is finite  such that
\begin{eqnarray}
\Psi_{\mu}(\varphi)=\inf_f\left\{ \Psi_{\mu}(f):  0<\overline{W}_{n+1-q}(f^*)<\infty, f\in L_1(\mu), f(0)=0 \right\}.
\end{eqnarray}
\end{pro}
\begin{proof}
  Since $f^{**}\leq f$ and $f^{***}=f^*$ for any function $f$,  it yields
\begin{eqnarray*}
\Psi_{\mu}(f^{**})&\leq&\Psi_{\mu}(f).
\end{eqnarray*}
We can just consider the infimum of $\Psi_{\mu}$ by restricting
our attention to the non-negative,  convex and  lower semi-continuous functions.

Let $\tilde{f}(x)= |x|$ and  set $c_{\mu}=\Psi_{\mu}(\tilde{f})$. Then $c_{\mu}$ is a finite real number. Let $\varphi_1,\varphi_2,\cdots$ be a minimizing sequence $\mu$-integrable convex functions with $\varphi_i(0)=0$ ($i=1,2,\cdots$), i.e.,
\begin{eqnarray}
\lim_{i\rightarrow\infty}\Psi_{\mu}(\varphi_i)=\inf_f\left\{ \Psi_{\mu}(f):  0<\overline{W}_{n+1-q}(f^*)<\infty, f(0)=0\right\}\leq c_{\mu}.
\end{eqnarray}
Furthermore, we may remove finitely many elements from the sequence $\{\varphi_i\}$ and assume that for all $i$,
\begin{eqnarray}
\Psi_{\mu}(\varphi_i)\leq c_{\mu}.
\end{eqnarray}
Since $\overline{W}_{n+1-q}(\varphi_i^*)>0$ for all $i$, therefore
\begin{eqnarray}\label{bound on the support of measure}
\frac{1}{|\mu|}\int_{\mathbb{R}^n}\varphi_i d\mu\leq c_{\mu}+1,
\end{eqnarray}
for all $i$.  Let $M$ be the interior of ${\rm conv}({\rm Supp}(\mu))$. By Lemma \ref{Cordero-ErausquinKlartag's lemma 16} and (\ref{bound on the support of measure}), we have
\begin{eqnarray}
\varphi_i (x_0)\leq (c_{\mu}+1)C_{\mu,x_0},
\end{eqnarray}
for all $i$ and any $x_0\in M$.
Theorem \ref{Rockafellar's result} ensures that there exists a subsequence $\{\varphi_{i_j}\}_{j=1,2,\cdots}$  that converges pointwise in $M$ to a convex function $\varphi: M\rightarrow \mathbb{R}$. Additionally, $\varphi$ is non-negative in $M$ with $\varphi(0)=0$.

 We extend the definition of $\varphi$ by setting $\varphi(x)=+\infty$ for $x\not\in \overline{M}$. We still need to define $\varphi$ for  $x\in\partial M$. Set
\begin{eqnarray}\label{value on boundary}
\varphi(x)=\lim_{\lambda\rightarrow 1^{-}}\varphi(\lambda x),\quad x\in\partial M.
\end{eqnarray}
Since the function $\lambda\rightarrow \varphi(\lambda x)$ is non-decreasing for $\lambda\in (0,1)$, hence the limit in (\ref{value on boundary}) always exists in $[0,+\infty]$. We need to show that $\varphi$ is $\mu$-integrable. For any $x\in \overline{M}$, the point $\lambda x \in M$ for $0<\lambda<1$, and by the pointwise convergence in $M$,
\begin{eqnarray}\label{pointwise convergence}
\varphi(\lambda x)=\lim_{j\rightarrow\infty} \varphi_{i_j}(\lambda x),
\end{eqnarray}
for any $x\in \overline{M}$. Since the measure $\mu$ does not support in a lower-dimensional subspace, so that $0\in {\rm Supp} (\mu)\subseteq  \overline{M}$, then from (\ref{pointwise convergence}), (\ref{bound on the support of measure}), for any $0<\lambda<1$, by the convexity and $\varphi_i(0)=0$,
\begin{eqnarray*}
\varphi_{i_j}(\lambda x)&\leq&\lambda\varphi_{i_j}( x)+(1-\lambda)\varphi_{i_j}(0)\\
&\leq&  \varphi_{i_j}( x),
\end{eqnarray*}
by Fatou's Lemma yields
\begin{eqnarray}\label{10}
\int_{\mathbb{R}^n}\varphi(\lambda x)d\mu(x)
&\leq&\liminf_{j\rightarrow\infty}\int_{\mathbb{R}^n}\varphi_{i_j}(\lambda x)d\mu(x)  \nonumber \\
&\leq&\liminf_{j\rightarrow\infty}\int_{\mathbb{R}^n}\varphi_{i_j}( x)d\mu(x).
\end{eqnarray}
Recall that we have $\varphi(\lambda x)\nearrow \varphi(x)$ as $\lambda\rightarrow 1^{-}$. From the monotone convergence theorem and (\ref{10}),
\begin{eqnarray*}\label{1}
\int_{\mathbb{R}^n}\varphi( x)d\mu(x)&=&\lim_{\lambda\rightarrow 1^{-}}\int_{\mathbb{R}^n}\varphi(\lambda x)d\mu(x)\\
&\leq&\liminf_{j\rightarrow\infty}\int_{\mathbb{R}^n}\varphi_{i_j}( x)d\mu(x)  \\
&<&+\infty.
\end{eqnarray*}
Hence, $\varphi$ is $\mu$-integrable.

By the continuity of $\varphi$ in $M$, we can choose a dense sequence $x_1,x_2,\cdots$ in $M$ such that
\begin{eqnarray*}
\varphi^*(y)&=&\sup_{x\in \mathbb{R}^n}\{\langle x, y\rangle-\varphi(x)\}\\
&=&\sup_{x\in M}\{\langle x, y\rangle-\varphi(x)\}\\
&=&\sup_{i\geq1}\{\langle x_i, y\rangle-\varphi(x_i)\},
\end{eqnarray*}
for $y\in \mathbb{R}^n$. For $j\geq1$, we set $\tilde{\varphi}_j^*(y)=\max_{1\leq i\leq j}\{\langle x_i, y\rangle-\varphi(x_i)\}$. We notice that $\tilde{\varphi}_j^*\nearrow \varphi^*$, then   the increment of function $t^{1-q}$ when $q\leq0$ and the monotone convergence theorem deduce
$$
\overline{W}_{n+1-q}(\varphi^*)
=\lim_{j\rightarrow\infty}\overline{W}_{n+1-q}(\tilde{\varphi}^*_j).
$$
For any $\varepsilon>0$, there exists $j_0$ such that
 $$
\left|\overline{W}_{n+1-q}(\varphi^*)
-\overline{W}_{n+1-q}(\tilde{\varphi}^*_{j_0})\right|\leq\frac{\varepsilon}{2}.
$$
By the fact $\varphi_{i_j}\rightarrow\varphi$ pointwise on the set $\{x_1,\cdots,x_{j_0}\}$,  then for sufficiently large $j$,  we have
\begin{eqnarray}
|\varphi^*_{i_j}(x)-\tilde{\varphi}^*_{j_0}(x)|\leq\frac{\varepsilon}{2},
\end{eqnarray}
for all $x\in \mathbb{R}^n$. Since $q<0$, then Minkowski's inequality deduces
\begin{eqnarray*}
\overline{W}_{n+1-q}(\varphi^*)
&\leq&\overline{W}_{n+1-q}(\tilde{\varphi}^*_{j_0})+\tfrac{\varepsilon}{2}\\
&\leq&\liminf_{j\rightarrow\infty}\left(\int_{\mathbb{R}^n}(\varphi^*_{i_j}(x)
+\tfrac{\varepsilon}{2})^{1-q}d\gamma_n(x)\right)^{\frac{1}{1-q}}+\tfrac{\varepsilon}{2}\\
&\leq&\liminf_{j\rightarrow\infty}\overline{W}_{n+1-q}(\varphi_{i_j}^*)+\varepsilon.
\end{eqnarray*}
This means 
\begin{eqnarray*}
 \overline{W}_{n+1-q}(\varphi^*) 
&\leq&\liminf_{j\rightarrow\infty} \overline{W}_{n+1-q}(\varphi_{i_j}^*)<+\infty , \end{eqnarray*}
and
\begin{eqnarray}\label{11}
-e^{-\overline{W}_{n+1-q}(\varphi^*)}
&\leq&\liminf_{j\rightarrow\infty}-e^{-\overline{W}_{n+1-q}(\varphi_{i_j}^*)}. \end{eqnarray}
Together  \eqref{11} with \eqref{10}, we have 
\begin{eqnarray}
\Psi_{\mu}(\varphi)\leq\liminf_{i\rightarrow\infty}\Psi_{\mu}(\varphi_i)
\end{eqnarray}
Therefore,  there exists a non-negative, $\mu$-integrable, convex function $\varphi$ with   $\varphi(0)=0$ such that
\begin{eqnarray*}
\Psi_{\mu}(\varphi)=\inf_f\left\{ \Psi_{\mu}(f):  f\in L_1(\mu), f(0)=0 \right\}.
\end{eqnarray*}
This completes our proof. 
\end{proof}

We are now in the position to prove the existence of solution to the  functional dual Minkowski problem. 

\vskip 0.3cm
\begin{thm}
Let $\mu$ be a non-zero finite Borel measure on  $\mathbb{R}^n$  with $\int_{\mathbb{R}^n}|x|d\mu(x)<\infty$  which is not supported in a lower-dimensional subspace and $q\leq0$.   There exists a non-negative  convex function $\varphi$ with $\varphi(0)=0$ and  $\int_{\mathbb{R}^n}\varphi(x)^{1-q}d\gamma_n(x)$  is finite  such that $\mu(\cdot)=\tau\widetilde{C}_q(\varphi,\cdot)$.
\end{thm}
\begin{proof}
The  proof  follows from  Lemma \ref{Existence} and Proposition \ref{existence of the minimization  problem}.
\end{proof}

\vskip 0.3cm
~~~
\subsection{Uniqueness of the dual functional  Minkowski problem}
~~~\vskip 0.3cm
In this subsection we will show that if the dual curvature measures of two convex functions are equal under some asumptions, then the convex functions are also equal almost everywhere in $\mathbb{R}^n$. The tool we will be used is the Minkowski inequality (\ref{Minkowski inequality}), similar to the uniqueness part of the $L_p$ Minkowski problem for convex bodies.

We  are ready to prove the uniqueness of the solution to the functional dual Minkowski problems.

\begin{thm}\label{Uniqueness}
Let $q\leq0$ and $\varphi_1,\varphi_2\in \mathcal{L}_o'$ with $\varphi_1(0)=\varphi_2(0)=0$. Assume that $\varphi_1$ is  an admissible perturbation for  $\varphi_2$ and $\varphi_2$ is also an admissible perturbation for  $\varphi_1$. 
  If $\widetilde{C}_{q-1}(\varphi_1,\cdot)=\widetilde{C}_{q-1}(\varphi_2,\cdot)$, then $\varphi_1(x)=\varphi_2(x-x_0)$ for some $x_0\in \mathbb{R}^n$ when $q=0$ and $\varphi_1(x)=\varphi_2(x)$ when $q<0$.
\end{thm}
\begin{proof}
Firstly notice that the equality $\widetilde{C}_{q-1}(\varphi_1,\cdot)=\widetilde{C}_{q-1}(\varphi_2,\cdot)$ implies $\overline{W}_{n+1-q}(\varphi_1)^{}
 =\overline{W}_{n+1-q}(\varphi_2)^{}$.
By the fact $\overline{W}_{n+1-q}(\varphi_1)^{}
 =\overline{W}_{n+1-q}(\varphi_2)^{}$ and  (\ref{Minkowski inequality}), we have
\begin{eqnarray}\label{Uniqueness1}
\widetilde{W}_{n+1-q}(\varphi_1,\varphi_2)\geq\widetilde{W}_{n+1-q}(\varphi_1,\varphi_1),
\end{eqnarray}
and
\begin{eqnarray}\label{Uniqueness2}
\widetilde{W}_{n+1-q}(\varphi_2,\varphi_1)\geq\widetilde{W}_{n+1-q}(\varphi_2,\varphi_2).
\end{eqnarray}

On the other hand,  since $\widetilde{C}_{q-1}(\varphi_1,\cdot)=\widetilde{C}_{q-1}(\varphi_2,\cdot)$, then Theorem \ref{integral formula} yields
\begin{eqnarray*}
\widetilde{W}_{n+1-q}(\varphi_1,\psi)=\widetilde{W}_{n+1-q}(\varphi_2,\psi),
\end{eqnarray*}
for any convex function $\psi$ with $\psi(0)=0$. In particular, taking $\psi=\varphi_1$ or $\psi=\varphi_2$, one sees that
\begin{eqnarray}\label{Uniqueness3}
\widetilde{W}_{n+1-q}(\varphi_1,\varphi_1)=\widetilde{W}_{n+1-q}(\varphi_2,\varphi_1),
\end{eqnarray}
and
\begin{eqnarray}\label{Uniqueness4}
\widetilde{W}_{n+1-q}(\varphi_2,\varphi_2)=\widetilde{W}_{n+1-q}(\varphi_1,\varphi_2).
\end{eqnarray}
Equalities (\ref{Uniqueness3}) and (\ref{Uniqueness4}) guarantee the inequalities (\ref{Uniqueness1}) and (\ref{Uniqueness2}) both hold with equality. Finally, the equality condition of (\ref{Minkowski inequality}) has make sure that $\varphi_1(x)=\varphi_2(x-x_0)$ for some $x_0\in \mathbb{R}^n$ when $q=0$ and $\varphi_1(x)=\varphi_2(x)$ when $q<0$.
\end{proof}

\vskip 1 cm


\end{document}